\documentclass[12pt,twoside,cd]{amsart}

\usepackage[]{geometry}
\numberwithin{equation}{section}

\usepackage{fancyhdr}
\usepackage{marvosym}

\usepackage{amsmath,amssymb,amsthm,amsrefs,mathtools,tikz, amscd, verbatim, graphicx, xcolor, color, mathrsfs, enumerate, dsfont, extarrows, eucal, floatrow, array, bm, tikz-cd}
\usetikzlibrary{shapes,matrix,patterns}
\usepackage[all]{xy}
\usepackage{pdflscape}

\usepackage[english]{babel}
\usepackage{graphicx}

\newcommand{\set}[1]{\left\lbrace #1 \right\rbrace}

\renewcommand{\tilde}[1]{\widetilde{#1}}

\newcommand{\Hom}{\operatorname{Hom}}
\newcommand{\trop}{\operatorname{trop}}
\newcommand{\spec}{\operatorname{Spec}}
\newcommand{\Bl}{\operatorname{Bl}}
\newcommand{\Gl}{\operatorname{Gl}}
\newcommand{\Sl}{\operatorname{Sl}}

\newcommand{\thespan}{\operatorname{span}}
\renewcommand{\hat}[1]{\widehat{#1}}

\newtheorem{theorem}{Theorem}[section]

\newtheorem{proposition}[theorem]{Proposition}

\theoremstyle{definition}
\newtheorem{remark}[theorem]{Remark}
\newtheorem{example}[theorem]{Example}
\newtheorem{definition}[theorem]{Definition}

\usepackage{chngcntr}
\counterwithin{figure}{section}

\title{Global Spherical Tropicalization via Toric Embeddings}
\author{Evan D. Nash}
\address{Department of Mathematics, The Ohio State University, 231 W 18th Ave, Columbus OH, 43210}
\email{nash.228@osu.edu}

\begin{document}
\begin{abstract}
The first steps in defining tropicalization for spherical varieties have been taken in the last few years. 
There are two parts to this theory: tropicalizing subvarieties of homogeneous spaces and tropicalizing their closures in spherical embeddings.
In this paper, we obtain a new description of spherical tropicalization that is equivalent to the other theories. This works by embedding in a toric variety, tropicalizing there, and then applying a particular piecewise projection map.
We use this theory to prove that taking closures commutes with the spherical tropicalization operation.
\end{abstract}
\maketitle
\small

\section*{Introduction}

In his thesis \cite{Vo}, Tassos Vogiannou introduces a notion of spherical tropicalization for a spherical homogeneous space $G/H$. Broadly speaking, this operation records the $G$-invariant divisors on the field of rational functions while the standard toric tropicalization instead records torus-invariant divisors. Kaveh and Manon in \cite{KM} recover this construction using a Gr\"obner theory for spherical varieties that they develop.
In \cite{Na}, we propose a means for tropicalizing embeddings of spherical homogeneous spaces.
There we define two constructions---one that mimics the theory for toric varieties developed by Kajiwara \cite{Ka} and Payne \cite{Pay} and one that extends the Gr\"obner theory definition of Kaveh and Manon---and show that they coincide.

In this paper, we give a third method for obtaining the tropicalization of a spherical embedding.
This method is global, by which we mean the tropicalization operation need only be applied once to a single object.
It is possible to globally tropicalize a toric variety by tropicalizing the quotient construction of a toric variety via the Cox ring.
See \cite{Pay}, Remark 3.5 or \cite{MS}, \textsection 6.2 for a description of this.
For details on the quotient construction, refer to \cite{Co1} and \cite{Co2} and for a comprehensive overview of Cox rings in general, see \cite{ADHL}.
Our aim is to mimic this construction for spherical varieties.
This is noteworthy because in the other two constructions of spherical tropicalization from \cite{Na}, the spherical variety is divided into simple $G/H$-embeddings, these are tropicalized separately, and then the tropicalizations are glued together.

The construction follows Gagliardi's work in \cite{Ga}, where he describes the Cox ring of a spherical embedding using its combinatorial data (see also \cite{Br07}).
We begin by embedding the spherical variety in a toric variety, which is an example of the embedding of a Mori dream space into a toric variety described by Hu and Keel \cite{HK}.
Once embedded in the toric variety, we can tropicalize there using the standard toric theory. 
Then applying a particular piecewise projection map will deliver the spherical tropicalization.

This process does not work for all spherical varieties. 
When the spherical variety does not have the $A_2$-property (Definition \ref{A2}), embedding into a toric variety is not possible.
The tropicalization of a spherical variety without the $A_2$-property can still be described using this theory, but it cannot be done globally.

The layout of this paper is as follows.
In \textsection \ref{background}, we review some of the theory of spherical varieties and their tropicalization.
In \textsection \ref{tropmorphisms}, we describe how to tropicalize a morphism between two spherical embeddings.
We explain a novel means of tropicalizing subvarieties of a spherical homogeneous spaces using toric embeddings in \textsection \ref{homspaces}.
Theorems \ref{V=G} and \ref{V=G2} prove that this new method coincides with the original theory of \cite{Vo}.
In \textsection \ref{toroidal} we show how to embed toroidal spherical varieties in toric varieties and in \textsection \ref{badcolors} we extend this to spherical embeddings with color.
Finally, \textsection \ref{extendedtrop} shows how to recover the extended tropicalization of an embedding by using the toric tropicalization. We also show in this section that taking the closure commutes with the tropicalization operation.

\subsection*{Acknowledgments} 
The author thanks Gary Kennedy for suggesting this direction of research and numerous discussions. Thanks also to Giuliano Gagliardi for providing input on several points.

\section{Background}\label{background}

We review some of the theory of spherical varieties.
For additional background, refer to \cite{LV}, \cite{Kn}, \cite{Pas}, or \cite{Pe}.
If $G$ is a connected reductive group, then the quotient $G/H$ is a \emph{(spherical) homogeneous space} if it is a normal variety containing a dense orbit of a Borel subgroup $B$.
A \emph{spherical embedding} $X$ is a normal $G$-variety with an open equivariant embedding $G/H \hookrightarrow X$.
We will also call $X$ a \emph{spherical variety} or a \emph{$G/H$-embedding} if we wish to highlight the underlying homogeneous space.
A spherical variety $X$ is modeled by a combinatorial object called a colored fan.
We recall the outline of the theory here; refer to the cited papers for further details. 

Denote by $\mathcal{X}$ the group of characters $B \rightarrow \mathbb{C}^*$ on $B$.
The set $\mathbb{C}(G/H)^{(B)}$ of $B$ semi-invariant rational functions on $G/H$ is defined by 
\[
\mathbb{C}(G/H)^{(B)} := \set{f \in \mathbb{C}(G/H)^\ast : \text{there exists } \chi_f \in \mathcal{X} \text{ such that } gf = \chi_f(g)f \text{ for all } g \in B },
\]
where the action of $g$ on $f$ is given by $gf(x) := f\left(g^{-1}x \right)$ for $x \in G/H$.
The subset $\mathcal{M} \subseteq \mathcal{X}$ of characters $\chi_f$ associated to $f \in \mathbb{C}(G/H)^{(B)}/\mathbb{C}^*$ is a lattice.
We write $\mathcal{N} := \Hom(\mathcal{M},\mathbb{Z})$ to denote its dual. 
Further, $\mathcal{M}_\mathbb{Q} := \mathcal{M} \otimes_\mathbb{Z} \mathbb{Q}$ and $\mathcal{N}_\mathbb{Q} := \mathcal{N} \otimes_\mathbb{Z} \mathbb{Q}$ are the associated $\mathbb{Q}$-vector spaces.
We write $\mathcal{V}(G/H)$ to denote the set of $G$-invariant valuations on $\mathbb{C}(G/H)^*$ that are trivial when restricted to $\mathbb{C}^*$.
A given $v \in \mathcal{V}(G/H)$ induces a homomorphism $\mathbb{C}(G/H)^{(B)}/\mathbb{C}^* \rightarrow \mathbb{Q}$, namely $f \mapsto v(f)$.
By identifying $f$ with its associated character $\chi_f$, this homomorphism is an element of $\mathcal{N}_\mathbb{Q}$.
Thus, we obtain a map $\mathcal{V}(G/H) \rightarrow \mathcal{N}_\mathbb{Q}$, which is an inclusion (\cite{LV}, Proposition 7.4). 
We will identify $\mathcal{V}(G/H)$ with its image in $\mathcal{N}_\mathbb{Q}$, which is a cone.
We call $\mathcal{V}(G/H)$ the \emph{valuation cone} of $G/H$ and write it as $\mathcal{V}$ when $G/H$ is understood.

We also obtain a \emph{palette} $\mathcal{D}$ of colors associated to $G/H$.
These are the prime divisors of $G/H$ that are $B$-stable but not $G$-stable.
Each $D \in \mathcal{D}$ induces a $G$-invariant valuation $v_D$ given by vanishing along $D$.
Thus, we obtain a map $\rho: \mathcal{D} \rightarrow \mathcal{N}$ given by $D \mapsto v_D$; it need not be injective.
We now have enough background to define colored cones and colored fans:
\begin{definition}
A \emph{colored cone} $(\sigma, \mathcal{F})$ consists of a cone $\sigma \subseteq \mathcal{N}_\mathbb{Q}$ and a subset $\mathcal{F} \subseteq \mathcal{D}$ such that:
\begin{itemize}
\item[(i)] $\sigma$ is generated by $\rho(\mathcal{F})$ and finitely many elements of $\mathcal{V}$;
\item[(ii)] $\operatorname{int}(\sigma) \cap \mathcal{V} \neq \emptyset$.
\end{itemize}
We call $(\sigma,\mathcal{F})$ \emph{strictly convex} if in addition $\sigma$ is strictly convex and $0 \notin \rho(\mathcal{F})$.
\end{definition}

\begin{definition}
A colored cone $(\tau,\mathcal{F}')$ is a \emph{face} of a colored cone $(\sigma, \mathcal{F})$ if $\tau$ is a face of $\sigma$ and $\mathcal{F}' = \mathcal{F} \cap \rho^{-1}(\tau)$. We write $\tau \preceq \sigma$ in this case.
\end{definition}

\begin{definition}
A \emph{colored fan} $\Sigma$ is a collection of colored cones such that:
\begin{itemize}
\item[(i)] If $(\sigma, \mathcal{F}) \in \Sigma$, the every face of $(\sigma, \mathcal{F})$ is in $\Sigma$;
\item[(ii)] For every $v \in \mathcal{V}$, there is at most one $(\sigma, \mathcal{F}) \in \Sigma$ with $v \in \operatorname{int}(\sigma)$.
\end{itemize}
We say $\Sigma$ is \emph{strictly convex} if in addition every colored cone in $\Sigma$ is strictly convex. 
\end{definition}

Note that the relative interiors of two colored cones in a colored fan may intersect nontrivially outside $\mathcal{V}$; we present an instance of this in Example \ref{colors_outside_val}.
There is a bijective correspondence between $G/H$-embeddings and strictly convex colored fans.
The colored cones in a strictly convex colored fan represent the $G$-orbits of an embedding $X$.
Strictly convex colored cones correspond to \emph{simple} $G/H$-embeddings, which have one closed $G$-orbit.
If the colored fan associated to a spherical variety has no colors, we say that the embedding is \emph{toroidal}.

Throughout this paper, we denote by $\nu: \mathbb{C}\{\!\{ t\}\!\} \rightarrow \mathbb{Q}$ the valuation that takes a Puiseux series to the lowest power of $t$ appearing with nonzero coefficient.
Given a closed subvariety $Y \subseteq G/H$, we define its spherical tropicalization as follows. For each $\mathbb{C}\{\!\{t\}\!\}$-point $\gamma \in Y\left( \mathbb{C}\{\!\{t\}\!\} \right)$, we obtain the following associated valuation $\nu_\gamma: \mathbb{C}(G/H)^* \rightarrow \mathbb{Q}$ in $\mathcal{V}$:
\[
f \mapsto \nu\left(\gamma^*(gf)\right) \text{ for } f \in \mathbb{C}(G/H)^* \text{ and } g \in G \text{ sufficiently general.}
\]
We must explain what we mean by sufficiently general $g$.
For $g$ in an open subset of $G$, the image of $\gamma$ is in the domain of $gf$.
Furthermore, the minimum of $\nu\left(\gamma^*(gf)\right)$ over all $g \in G$ is met on an open subset.
A sufficiently general $g$ lies in both of these open subsets.
The tropicalization $\trop_G(Y) \subseteq \mathcal{V}$ is the collection of $\nu_\gamma$ for all $\gamma \in Y\left( \mathbb{C}\{\!\{t\}\!\} \right)$. 
Note that the subscript $G$ records which group action we are tropicalizing with respect to and observe that $\trop_G(G/H) = \mathcal{V}(G/H)$.

We briefly describe the extended tropicalization construction for spherical embeddings from \cite{Na}.
If $X$ is a $G/H$-embedding and $\mathcal{O} \subseteq X$ is a $G$-orbit, then $\mathcal{O}$ is itself a homogeneous space with respect to the action of $G$.
Thus, $\mathcal{O}$ has a valuation cone $\mathcal{V}(\mathcal{O})$.
As a set, the spherical tropicalization is 
\[
\trop_G(X) := \bigsqcup_{\mathcal{O} \subseteq X} \mathcal{V}(\mathcal{O}).
\]
Note that the subscript $G$ is used to record which group action we are considering.
The subset $\mathcal{V}(\mathcal{O}) \in \trop_G(X)$ can be thought of as extended valuations on $\mathbb{C}[G/H]$, i.e. semigroup homomorphisms $\mathbb{C}[G/H] \rightarrow \mathbb{Q} \cup \set{\infty}$.
Which elements of $\mathbb{C}[G/H]$ that are allowed to be sent to $\infty$ is determined by which functions in $\mathbb{C}[G/H]$ vanish along $\mathcal{O}$.
The topology on $\trop_G(X)$ in essence adds on these extended valuations by considering them as limit points of finitely-valued $\mathbb{Q}$-valuations.
If $Y \subseteq X$ is a closed subvariety, we obtain its tropicalization $\trop_G(Y)$ by separately tropicalizing its intersection $Y \cap \mathcal{O}$ with each $G$-orbit $\mathcal{O}$ of $X$.
We define
\[
\trop_G(Y) := \bigsqcup_{\mathcal{O} \subseteq X} \trop_G(Y \cap \mathcal{O}) \subseteq \trop_G(X)
\] 
and apply the subspace topology inherited from $\trop_G(X)$.

We can identify the tropicalization of a simple $G/H$-embedding consisting of one maximal colored cone $(\sigma,\mathcal{F})$ with a set of semigroup homomorphisms $\Hom^{\mathcal{V}}(\sigma^\vee \cap \mathcal{M}, \overline{\mathbb{Q}}) \subseteq \Hom(\sigma^\vee \cap \mathcal{M}, \overline{\mathbb{Q}})$.
This set consists of all semigroup homomorphisms $\mu \in \Hom(\sigma^\vee \cap \mathcal{M}, \overline{\mathbb{Q}})$ such that $\mu^{-1}(\mathbb{Q}) = \tau^\perp \cap \sigma^\vee \cap \mathcal{M}$ where $(\tau,\mathcal{F}')$ is a face of $(\sigma,\mathcal{F})$. 

We establish some notational conventions.
Throughout, $X$ denotes a spherical embedding, $Y \subseteq G/H$ denotes a closed subvariety of a homogeneous space, and $Z$ denotes a toric variety.
The combinatorial data associated to spherical varieties is written in non-calligraphic font when referring to toric varieties.
That is, a toric variety has the lattices $M$ and $N$ rather than $\mathcal{M}$ and $\mathcal{N}$.
Finally, we write $\overline{\mathbb{Q}} := \mathbb{Q} \cup \set{\infty}$.

\section{Tropicalizing Morphisms}\label{tropmorphisms}

In this section, we describe how a morphism between spherical varieties induces a morphism between their tropicalizations. Our main result is that these morphisms commute with tropicalization (Proposition \ref{commute}), which extends Corollary 6.2.17 of \cite{MS}.
This fact will be used later; we prove it here because it is independent of the other theory in this paper.

Let $X$ and $X'$ respectively be a $G/H$-embedding and a $G/H'$-embedding.
If $\Phi: X \rightarrow X'$ is a $G$-morphism induced by a dominant $G$-equivariant morphism $G/H \rightarrow G/H'$, then there is an induced continuous map $\text{trop}(\Phi): \text{trop}(X) \rightarrow \text{trop}(X')$.

Before describing the map, we collect some facts about morphisms between spherical varieties as stated in \textsection 4 of \cite{Kn}.
First, a $G$-equivariant dominant morphism $\Phi: G/H \rightarrow G/H'$ between spherical homogeneous spaces induces an injection $\Phi^\ast: \mathcal{M}(G/H') \hookrightarrow \mathcal{M}(G/H)$.
This in turn induces a surjection $\Phi_*: \mathcal{N}(G/H) \rightarrow \mathcal{N}(G/H')$, which restricts to $\Phi_*: \mathcal{V}(G/H) \rightarrow \mathcal{V}(G/H')$.
Note that $\Phi^*$ is defined by $f \mapsto f \circ \Phi$ and $\Phi_*$ is defined by $\mu \mapsto \mu \circ \Phi^*$.

Let $\mathcal{O}_i \subseteq X$ be a $G$-orbit with valuation cone $\mathcal{V}_i$.
Because $\Phi$ is $G$-equivariant, it takes orbits of $X$ to orbits of $X'$, so $\Phi(\mathcal{O}_i) \subseteq \mathcal{O}_i'$ for some orbit $\mathcal{O}_i'$ of $X'$.
Let $\mathcal{V}_i'$ denote the valuation cone of $\mathcal{O}_i'$.
By restriction, we have a map $\mathcal{O}_i \rightarrow \mathcal{O}_i'$ and hence an induced map $\mathcal{V}(\mathcal{O}_i) \rightarrow \mathcal{V}(\mathcal{O}_i')$.
Since this holds for each $G$-orbit of $X$, we can define a map $\trop(\Phi): \bigsqcup_i \mathcal{V}_i \rightarrow \bigsqcup_j \mathcal{V}_j'$ by taking the disjoint union of the pushforwards.

Suppose that we have a morphism $\Phi: G/H \rightarrow G'/H'$ of homogeneous spaces that is equivariant with respect to some surjective homomorphism of algebraic groups $\varphi: G \rightarrow G'$.
We would like to define the tropicalization of $\Phi$ in this setting where $G$ is not necessarily equal to $G'$.
The homomorphism $\varphi$ gives an action of $G$ on $G'/H'$; we show that this action makes $G'/H'$ a spherical homogeneous space with respect to $G$.

%

By choosing an appropriate basepoint, we may assume that $\Phi(H) = H'$.
Equivariance then implies that for any $g \in G$, $\Phi(gH) = \varphi(g)H'$.
Note that we also have $\phi(H) \leq H'$.
Indeed, if $h \in H$, we have $H' = \Phi(H) = \Phi(hH) = \varphi(h)H'$,
so $\varphi(h) \in H'$.
Then $\Phi$ can be factored as the natural projection $G/H \rightarrow G/\varphi^{-1}(H')$ followed by $\overline{\varphi}: G/\varphi^{-1}(H') \rightarrow G'/H'$ given by $g\varphi^{-1}(H') \mapsto \varphi(g)H'$.
By Zariski's Main Theorem (see for example \cite{Mu}, \textsection III.9), $G/\varphi^{-1}(H') \cong G'/H'$ as varieties, so $\Phi: G/H \rightarrow G'/H'$ may be realized as the projection $G/H \rightarrow G/\varphi^{-1}(H')$, and we may define $\trop(\Phi): \trop_G \left(G/H \right) \rightarrow \trop_G\left( G/\varphi^{-1}(H') \right)$ as we did when $G' = G$ and $\varphi$ was the identity.

\begin{proposition}\label{trop map continuous}
Let $X$ and $X'$ be a $G/H$-embedding and a $G'/H'$-embedding, respectively.
Suppose $\Phi: X \rightarrow X'$ is a dominant morphism equivariant with respect to a surjective homomorphism $\varphi: G \rightarrow G'$.
Then $\trop(\Phi)$ is continuous.
\end{proposition}
\begin{proof}
By the argument preceding the proposition, we may assume that $G = G'$ and $\varphi$ is the identity map.

Let $\mu \in \trop_G(\mathcal{O})$ where $\mathcal{O}$ is a $G$-orbit in $X$ and let the $G$-orbit $\Phi(\mathcal{O})$ correspond to a colored cone $(\sigma', \mathcal{F}')$.
Then there is a sequence $\set{\mu_\ell}_{\ell = 1}^\infty \subset \trop_G(G/H)$ of valuations such that $\lim_{\ell \rightarrow \infty} \mu_\ell = \mu$ in the topology on $\trop_G(X)$.
Let $f \in (\sigma')^\vee \cap \mathcal{M}'$ be an arbitrary $B$ semi-invariant rational function on the orbit $\Phi(\mathcal{O})$.
Then:
\[
\lim_{\ell \rightarrow \infty} (\trop(\Phi)(\mu_\ell))(f) = \lim_{\ell \rightarrow \infty} (\mu_\ell \circ \Phi^*)(f) = \lim_{\ell \rightarrow \infty} \mu_\ell \circ (f \circ \Phi) = \mu \circ (f \circ \Phi) = (\trop(\Phi)(\mu))(f)
\]
It follows that $\trop(\Phi)(\mu_\ell) \rightarrow \trop(\Phi)(\mu)$ as $\ell \rightarrow \infty$, and the claim follows.
\end{proof}

\begin{proposition} \label{commute} Let $X$ and $X'$ be spherical embeddings with respect to groups $G$ and $G'$, respectively. Let $\Phi: X \rightarrow X'$ be a dominant morphism equivariant with respect to a surjective homomorphism $\varphi: G \rightarrow G'$ such that $G/\varphi^{-1}(H) \cong G'/H'$ as varieties. Then if $Y \subseteq X$ is a subvariety, $\trop_{G'}(\Phi(Y)) = \trop(\Phi) (\trop_G(Y))$.
\end{proposition}
\begin{proof}
As before, we may assume that $G = G'$ and $\varphi$ is the identity map.

Let $\mathcal{O} \subseteq X$ be an arbitrary $G$-orbit of $X$ and consider the restriction of $\trop(\Phi)$ to the valuation cone $\mathcal{V}(\mathcal{O})$ associated to $\mathcal{O}$.
By definition of $\trop(\Phi)$, the image of $\mathcal{V}(\mathcal{O})$ lies in the valuation cone $\mathcal{V}(\mathcal{O}')$ of some orbit $\mathcal{O}' \subseteq X'$.
We will prove the statement in the case that $Y \subseteq \mathcal{O}$ is a subvariety of the spherical homogeneous space $\mathcal{O}$.
The full statement follows directly by considering the individual intersections of a subvariety with each orbit.

Let $\gamma: \spec{\mathbb{C}\{\!\{t\}\!\}} \rightarrow Y$ be a $\mathbb{C}\{\!\{t\}\!\}$-point of $Y$.
Then $\Phi \circ \gamma$ is a $\mathbb{C}\{\!\{t\}\!\}$-point of $\Phi(Y)$ and all $\mathbb{C}\{\!\{t\}\!\}$-points of $\Phi(Y)$ arise in this way.
This is because $\Phi: \mathcal{O} \rightarrow \mathcal{O}'$ is surjective as it is $G$-equivariant with respect to the surjective map $\varphi$.
Thus, $\trop_{G}(\Phi(Y))$ is defined as follows:
\[
\trop_{G}(\Phi(Y)) := \set{ \begin{array}{c}
k(\Phi(Y))^{(B)} \rightarrow \mathbb{Q} \\
f \mapsto \nu((\Phi \circ \gamma)^\ast(gf)) 
\end{array}
: \gamma \in Y(\mathbb{C}\{\!\{t\}\!\})},
\]
where $g$ is chosen to be sufficiently general for a given $f$.
The tropicalization of $Y$ is 
\[
\trop_G(Y) := \set{ \begin{array}{c}
k(Y)^{(B)} \rightarrow \mathbb{Q} \\
f \mapsto \nu(\gamma^\ast(gf)) 
\end{array} : \gamma \in Y(\mathbb{C}\{\!\{t\}\!\})},
\]
again for sufficiently general $g$.
Applying $\trop(\Phi)$ means taking $\nu_\gamma \mapsto \nu_\gamma \circ \Phi^*$, so this gives us $\trop(\Phi)(\trop_G(Y))$:
\[
\trop(\Phi)(\trop_G(Y)) := \set{ \begin{array}{c}
k(\Phi(Y))^{(B)} \rightarrow \mathbb{Q} \\
f \mapsto \nu(\gamma^\ast(g\Phi^\ast(f))) 
\end{array} : \gamma \in Y(\mathbb{C}\{\!\{t\}\!\})}.
\]
To prove the proposition, we need to show that $\nu((\Phi \circ \gamma)^*(gf)) = \nu(\gamma^*(g\Phi^*(f)))$ for all $f \in k(\Phi(Y))^{(B)}$ where $g$ is chosen to be sufficiently general for both $f$ and $\Phi^*(f)$.
Note that $(\Phi \circ \gamma)^* = \gamma^* \circ \Phi^*$, so it will be sufficient to show that $\Phi^*(gf) = g\Phi^*(f)$.
For $x \in Y$ we have the following, using the fact that $\Phi$ is $G$-equivariant:
\[
\Phi^\ast(gf)(x) = (gf \circ \Phi)(x) = f(g^{-1}\Phi(x)) = (f \circ \Phi)(g^{-1}x) = g(f \circ \Phi)(x) = g\Phi^*(f)(x). 
\]
Thus, $\Phi^\ast(gf) = g\Phi^*(f)$ and so the proposition holds.
\end{proof}

\section{Tropicalizing Homogeneous Spaces via Toric Embeddings}\label{homspaces}

In \cite{Ga}, Gagliardi proves a theorem relating the valuation cone of $G/H$ to tropicalization. Inspired by his result, we  find an alternate means for tropicalizing a subvariety of a spherical homogeneous space. The purpose of this section is to set up Gagliardi's theory and show how to recover from it Vogiannou's spherical tropicalization (Theorem \ref{V=G}).

Throughout, we will always assume that $G$ is of simply connected type, which means that $G = G^{ss} \times C$ where $G^{ss}$ is semi-simple simply connected and $C$ is a torus.
Every connected reductive group has a finite covering by a group of simply connected type with the same embedding, so this is not a restrictive assumption.

We work initially over a spherical homogeneous space $G/H$ with trivial divisor class group; the case of non-trivial divisor class group will be handled later.
We will describe how to associate a toric variety $Z_0$ to $G/H$.
Each color in the palette of $G/H$ is a prime divisor $D_i = V(f_i)$.
The orbit of $f_i$ in $\mathbb{C}(G/H)$ under the action of $G$ spans a $G$-module of some rank $s_i$.
We choose a basis $\set{f_{i1} := f_i,f_{i2},\ldots,f_{is_i}} \subseteq G \cdot f_i$ for this $G$-module.
Further, $\Gamma\left(G/H,\mathcal{O}^*_{G/H} \right)/\mathbb{C}^*$ is a finitely generated free abelian group with basis $\{g_k\}_{k=1}^m$.
The characters associated to $f_i$ and $g_k$ span the lattice $\mathcal{M}$ of characters of $B$ semi-invariant rational functions $k(G/H)^{(B)}$.
Denote these characters by $v_i^*$ and $w_k^*$, respectively.
Define a toric variety
\[
Z_0 := (\mathbb{C}^{s_1} \setminus \set{0}) \times \cdots \times (\mathbb{C}^{s_r} \setminus \set{0}) \times (\mathbb{C}^*)^m.
\] 
with coordinates $f_{ij}$ and $g_k$, matching the basis of $\mathcal{M}$. Denote the lattice of cocharacters of $Z_0$ by $N \cong \mathbb{Z}^{s_1 + \cdots + s_r + m}$ with basis $\{v_{11},v_{12},\ldots,v_{rs_r},w_1,\ldots,w_m \}$.
Write $\mathbb{T} \cong (\mathbb{C}^*)^{s_1 + \cdots + s_r + m}$ for the dense torus in $Z_0$.
The inclusion $\iota: G/H \hookrightarrow Z_0$ defined by $x \mapsto (f_{ij}(x))_{i,j} \times (g_k(x))_k$ is a closed embedding.
This induces a natural action of $G$ on $Z_0$ commuting with $\iota$.
The coordinate ring of $Z_0$ has coordinates $S_{ij}$ for $1 \leq i \leq r$ and $1 \leq j \leq s_i$ and $T_k$ for $1 \leq k \leq m$.
The inclusion $\iota$ is dual to the map $\Psi: \mathbb{C}[Z_0] \rightarrow \mathbb{C}[G/H]$ defined by $S_{ij} \mapsto f_{ij}$ and $T_k \mapsto g_k$.
Let $\mathfrak{p}$ denote the kernel of $\Psi$ so that $\mathbb{C}[Z_0]/\mathfrak{p} \cong \mathbb{C}[G/H]$.
Finally, the lattice $\mathcal{N}$ dual to $\mathcal{M}$ has basis $\{v_1,\ldots,v_r,w_1,\ldots,w_m \}$ and we define an inclusion $\mathcal{N} \hookrightarrow N$ by $v_i \mapsto v_{i1} + \cdots + v_{is_i}$ and $w_k \mapsto w_k$.

Theorem \ref{GagVal} shows how the embedding $G/H \hookrightarrow Z_0$ can be related to the valuation cone of $G/H$:

\begin{theorem}[\cite{Ga}, Theorem 1.7]\label{GagVal}
$\mathcal{V}(G/H) = \trop_\mathbb{T}(G/H \cap \mathbb{T}) \cap \mathcal{N}_\mathbb{Q}$.
\end{theorem}

We can also write this as $\trop_G(G/H) = \trop_\mathbb{T}(G/H \cap \mathbb{T}) \cap \mathcal{N}_\mathbb{Q}$.
In fact, because $\mathcal{N}_\mathbb{Q} \subseteq N_\mathbb{Q} = \trop_\mathbb{T}(\mathbb{T})$, we may write $\trop_G(G/H) = \trop_\mathbb{T}(G/H) \cap \mathcal{N}_\mathbb{Q}$.
Note that $\trop_\mathbb{T}(G/H)$ is potentially ill-defined if one does not have a notion of extended tropicalization since $G/H$ is not necessarily contained in $\mathbb{T}$.

Theorem \ref{GagVal} does not extend to subvarieties $Y \subseteq G/H$.
That is, $\trop_G(Y) \neq \trop_\mathbb{T}(Y) \cap \mathcal{N}_\mathbb{Q}$ in general.
Ultimately, accounting for subvarieties of $G/H$ requires a new result independent of Theorem \ref{GagVal}, which we describe now.
Note that 
\[
\trop_{\mathbb{T}}(Z_0) = \left(\overline{\mathbb{Q}}^{s_1} \setminus \set{\infty} \right) \times \cdots \times \left(\overline{\mathbb{Q}}^{s_r} \setminus \set{\infty} \right) \times \mathbb{Q}^m
\]
and define a map $\psi: \trop_{\mathbb{T}}(Z_0) \rightarrow \mathcal{N}_\mathbb{Q}$ as follows:
\begin{align*}
\psi: \trop_{\mathbb{T}}(Z_0) & \rightarrow \mathcal{N}_\mathbb{Q} \\
(a_{11},\ldots,a_{1s_1},a_{21},\ldots,a_{rs_r},b_1,\ldots,b_m) & \mapsto \left(\min_{1 \leq j \leq s_1} \set{a_{1j}},\ldots,\min_{1 \leq j \leq s_r}\set{a_{rj}},b_1,\ldots,b_m \right).
\end{align*}

The following theorem is the core idea of this paper. The further work ultimately simply expands on this result. 

\begin{theorem}\label{V=G}
If $Y \subseteq G/H$ is a closed subvariety and $G/H$ has trivial divisor class group, then
$\trop_G(Y) = \psi\left(\trop_\mathbb{T}(Y) \right)$.
\end{theorem}
\begin{proof}
We first consider the left-to-right inclusion.
Let $\gamma$ be a $\mathbb{C}\{\!\{t\}\!\}$-point of $Y \subseteq G/H$.
Vogiannou's definition gives us a $G$-invariant valuation $\nu_\gamma \in \trop_G(Y) \subseteq \mathcal{N}_\mathbb{Q}$ defined by $f \mapsto \nu(\gamma^*(gf))$ for sufficiently general $g \in G$.
The dual lattice $\mathcal{M}$ to $\mathcal{N}$ is spanned by characters associated to the $f_i$ and $g_k$, so we only need to know how $\nu_\gamma$ behaves on these functions to completely determine it as an element of $\mathcal{N}_\mathbb{Q}$.
Consider $\nu_\gamma(f_i)$ for some $i$, suppose that $g \in G$ is sufficiently general so that $\nu_\gamma(f_i) = \nu(\gamma^*(gf_i))$, and write $gf_i = \sum_{j=1}^{s_i} a_jf_{ij}$ where $a_j \in \mathbb{C}$.
Generically, the minimum $\min_j\set{\nu(a_jf_{ij}(\gamma))}$ is met at only one monomial, so we may write $\nu_\gamma(f_i) = \nu(a_jf_{ij}(\gamma))$ for some $j$.
Note that $a_jf_{ij}(\gamma) \neq 0$ as otherwise $f_i$ would be zero on all of $G/H$.
Further, observe that by this argument $\nu_\gamma(f_{ij}) = \nu_\gamma(f_{ik})$ for all $j$ and $k$, so $\trop_{G}(Y) \subseteq \mathcal{N}_\mathbb{Q} \subseteq N_\mathbb{Q}$.
Applying $\psi$ to the $\mathbb{T}$-invariant valuation induced by $\gamma$ gives $\nu_\gamma$, defined by $f_i \mapsto \nu(a_jf_{ij}(\gamma))$.

By Proposition 1.3 of \cite{KKV}, the $g_k$ are all $G$ eigenvectors, so it follows that the $\mathbb{T}$-invariant and $G$-invariant valuations induced by $\gamma$ act identically on the $g_k$.
This implies the first inclusion.

Conversely, suppose that $\gamma$ is a $\mathbb{C}\{\!\{t\}\!\}$-point of $Y \subseteq Z_0$ and consider $\psi\left(\tilde{\nu}_\gamma \right)$, where $\tilde{\nu}_\gamma$ is the $\mathbb{T}$-invariant valuation induced by $\gamma$.
For fixed $i$, suppose that $\min_j\set{\nu(f_{ij}(\gamma))} = \nu(f_{ik}(\gamma))$ for some $k$.
Note that by definition of $Z_0$ we cannot have $f_{ij}(\gamma) = 0$ for all $j$, so the minimum is finite and thus the image of $Y$ under $\psi$ is contained in $\mathcal{N}_\mathbb{Q}$.
For generic $g \in G$, $gf_i$ has a nonzero coefficient $a_{ik}$ on $f_{ik}$, so it follows that $\nu_\gamma(f_i) = \nu(a_{ik}f_{ik}(\gamma)) = \psi(\tilde{\nu}_\gamma)(f_i)$, where $\nu_\gamma$ is the $G$-invariant valuation induced by $\gamma$.
Because the $g_k$ are $G$-eigenvectors, $\tilde{\nu}_\gamma(g_k) = \nu_\gamma(g_k)$ for all $k$ and so $\psi\left( \tilde{\nu}_\gamma \right) = \nu_\gamma \in \trop_G(Y)$.
\end{proof}

\begin{example}\label{punctured affine in toric}
Let $G = \Gl_2$ and $H$ be the subgroup of upper triangular matrices with ones on the diagonal.
Then $G/H \cong \mathbb{C}^2 \setminus \set{0}$ where the action of $G$ is given by matrix multiplication on a vector $(x \; y)^T$. 
The Borel subgroup $B$ of upper triangular matrices has an open orbit $D(y)$, the principal open set where $y$ doesn't vanish.
There is one color $V(y) \subset G/H$ and $\mathcal{V} = \mathcal{N}_\mathbb{Q} \cong \mathbb{Q}$.

In this setting, $m = 0$, $r = 1$, $s_1 = 2$, and we may write $f_{11} := y$ and $f_{12} := x$. 
The lattice $\mathcal{N}$ is spanned by a single ray $v_1$. The left-hand side of Figure \ref{GagAff} shows the vectors $v_1$, $v_{11}$, and $v_{12}$. Then $Z_0 \cong \mathbb{C}^2 \setminus \set{0}$ and the inclusion $G/H \hookrightarrow Z_0$ is an isomorphism given by switching coordinates: $(a,b) \mapsto (b,a)$.
The tropicalization $\trop_\mathbb{T}(G/H)$ is $\overline{\mathbb{Q}}^2 \setminus \set{\infty}$ and the inclusion $\mathcal{N}_\mathbb{Q} \hookrightarrow N_\mathbb{Q}$ is the diagonal map $\mathbb{Q} \hookrightarrow \mathbb{Q}^2$. 
Figure \ref{GagAff} illustrates Theorem \ref{V=G} in the case where $Y = G/H$. In the figure, $\trop_G(G/H) = \mathcal{V}$ appears as a dotted line and we conclude that $\mathcal{V} = \mathcal{N}_\mathbb{Q}$ since $\trop_\mathbb{T}(\mathbb{C}^2 \setminus \set{0})$ surjects onto $\mathcal{N}_\mathbb{Q}$ via $\psi$.
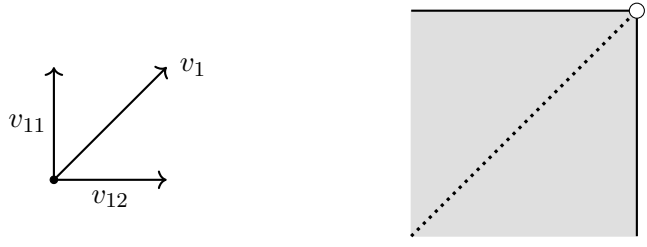
\begin{figure}[H]
\begin{tikzpicture}
\draw[fill] (9.25,.75) circle(.05);
\draw[thick,->] (9.25,.75)--(10.75,.75);
\draw[thick,->] (9.25,.75)--(9.25,2.25);
\draw[thick,->] (9.25,.75)--(10.75,2.25); 

\draw (8.9,1.5) node[]{$v_{11}$};
\draw (10,.5) node[]{$v_{12}$};
\draw (11.1,2.25) node[]{$v_{1}$};

\draw[gray!25,fill=gray!25] (14,0) rectangle (17,3);
\draw[thick] (14,3)--(17,3)--(17,0);
\draw[very thick,dotted] (14,0)--(17,3);
\draw[fill=white] (17,3) circle(.1); 
\end{tikzpicture}
\caption{The vectors $v_1$, $v_{11}$, and $v_{12}$ of Example \ref{punctured affine in toric} (left) and the tropicalization of the punctured plane with $\mathcal{N}_\mathbb{Q}$ embedded as a dotted line (right)}
\label{GagAff}
\end{figure}
This result allows us to recover a description due to Vogiannou of the tropicalization of subvarieties of $G/H$ in this setting (\cite{Vo}, Example 3.10).
Vogiannou considers a curve $C$ in $G/H = \mathbb{C}^2 \setminus \set{0}$ given by a polynomial $f(x,y)$.
He shows that $\trop_G(C)$ is a ray oriented to the left in $\mathcal{V} \cong \mathbb{Q}$ when the constant term of $f$ is nonzero and is all of $\mathcal{V}$ when the constant term is zero.
The condition that the constant term of $f(x,y)$ is zero is equivalent to saying that the zero vector is contained in the closure of $C$ in $\mathbb{C}^2$.
By Proposition 6.3.5 of \cite{MS}, this occurs if and only if $\trop_\mathbb{T}(C)$ intersects the interior of the first quadrant.
In this case, $\trop_\mathbb{T}(C)$ projects onto the right-hand ray of $\mathcal{V}$ under the map $\psi$, so we conclude that $\trop_G(C) = \mathcal{V}$. 
\end{example}

\begin{example}\label{speciallinear}
Let $G := \Sl_2 \times \Sl_2$, $H$ be the diagonal subgroup, and $B$ consist of ordered pairs of upper and lower triangular matrices.
Then $G/H \cong \Sl_2$ with coordinates $x_{ij}$ for $i,j = 1,2$ and the action of $G$ is $(g,h) \cdot (x_{ij}) = g(x_{ij})h^{-1}$.
The only $B$ semi-invariant rational function is $f_1 := x_{22}$ with associated character $\chi: B \rightarrow \mathbb{C}^*$ defined by $((a_{ij}),(b_{ij})) \mapsto a_{22}^{-1}b_{22}$. 
The orbit of $x_{22}$ under the action of $G$ is 
\[
G \cdot x_{22} = \set{-g_{21}h_{12}x_{11} + g_{21}h_{11}x_{12} - g_{22}h_{12}x_{21} + g_{22}h_{11}x_{22} : (g_{ij}),(h_{ij}) \in \Sl_2 }
\]
It follows that $G \cdot f_1$ has rank four and we may choose as our basis $f_{11} := f_1 = x_{22}$, $f_{12} := x_{21}$, $f_{13} := x_{12}$, and $f_{14} := x_{11}$.
Finally, $\Gamma\left(G/H,\mathcal{O}_{G/H}^*\right)/\mathbb{C}^*$ is trivial, so $Z_0 \cong \mathbb{C}^4 \setminus \set{0}$.
The inclusion $G/H \hookrightarrow Z_0$ is then given by $(x_{ij}) \mapsto (x_{22},x_{21},x_{12},x_{11})$.

The image of $G/H$ in $Z_0$ is $V(f_{11}f_{14} - f_{12}f_{13} - 1)$.
The tropicalization of this variety is the set of extended valuations $\mu$ such that the minimum 
\[
\min\set{\mu(f_{11}) + \mu(f_{14}), \mu(f_{12}) + \mu(f_{13}), 0}
\]
is met more than once.
This forces $\mu(f_{1j}) \leq 0$ for some $j$.
Applying $\psi$ therefore takes us to the ray in $\mathcal{N}_\mathbb{Q} \cong \mathbb{Q}$ spanned by the valuation $f_1 \mapsto -1$, so $\mathcal{V}$ is a half-space.
\end{example}

We must still consider the case that the spherical homogeneous space has non-trivial divisor class group.
Following \cite{Ga}, we will use bold-faced characters $\bm{G}$ and $\bm{H}$ when the spherical homogeneous space $\bm{G}/\bm{H}$ does not have trivial divisor class group.
In general, the spherical data associated to $\bm{G}/\bm{H}$ is notated as with $G/H$ but bold-faced.
We still assume that $\bm{G} = \bm{G}^{ss} \times \bm{C}$ is of semisimple type.
For each $\bm{D}_i$ in the palette $\bm{\mathcal{D}} := \set{\bm{D}_1,\ldots,\bm{D}_r}$, we consider the pullback of $\bm{D}_i$ under the projection map $\bm{G} \rightarrow \bm{G}/\bm{H}$.
By results in \cite{Br07}, there is a unique $\bm{f}_i \in \mathbb{C}[\bm{G}]$ such that $V(\bm{f}_i)$ cuts out the pre-image of $\bm{D}_i$, $\bm{f}_i$ is $\bm{C}$-invariant, and $\bm{f}_i(1) = 1$.
Then $\bm{H}$ acts from the right on $\bm{f}_i$ with character $\bm{\chi}_i \in \mathcal{X}(\bm{H})$. We will come back to these $\bm{f}_i$ in \textsection \ref{badcolors}.

Gagliardi defines $G := \bm{G} \times (\mathbb{C}^*)^{\bm{\mathcal{D}}}$ and 
\[
H := \set{(h,\bm{\chi}_1(h),\ldots,\bm{\chi}_r(h)) : h \in \bm{H}}.
\]
The Borel subgroup is $B := \bm{B} \times (\mathbb{C}^*)^{\bm{\mathcal{D}}}$.
There is a natural isomorphism $H \cong \bm{H}$ given by projection to the first coordinate and a natural morphism $\bm{\pi}: G/H \rightarrow \bm{G}/\bm{H}$ with the induced surjective map $\bm{\pi}_*: \mathcal{N}_\mathbb{Q} \rightarrow \bm{\mathcal{N}}_\mathbb{Q}$.
The $B$-stable prime divisors in the palette $\mathcal{D}$ of $G/H$ are precisely the pullbacks under $\bm{\pi}$ of the colors in $\bm{\mathcal{D}}$.
The character group of $(\mathbb{C}^*)^{\bm{\mathcal{D}}}$ is isomorphic to $\mathbb{Z}^{\bm{\mathcal{D}}}$, and we denote its basis by $\set{\eta_1,\ldots,\eta_r}$.
Gagliardi shows that $G/H$ has trivial divisor class group (\cite{Ga}, Corollary 3.2) and that $\mathcal{V} = \bm{\pi}_*^{-1}(\bm{\mathcal{V}})$ (\cite{Ga}, Proposition 3.3).

Theorem \ref{V=G} can now be easily generalized:
\begin{theorem}\label{V=G2}
If $\bm{Y} \subseteq \bm{G}/\bm{H}$ is a closed subvariety, then
\[
\trop_{\bm{G}}(\bm{Y}) = (\trop(\bm{\pi}) \circ \psi) \left( \trop_{\bm{\mathbb{T}}}\left(\bm{\pi}^{-1}\left(\bm{Y} \right)\right) \right).
\]
\end{theorem}
\begin{proof}
By Theorem \ref{V=G}, this is equivalent to proving $\trop_{\bm{G}}(\bm{Y}) = \trop(\bm{\pi}) \left( \trop_G\left(\bm{\pi}^{-1}(\bm{Y}) \right) \right)$, which follows directly from Proposition \ref{commute}.
\end{proof}

\begin{example}
Suppose $\bm{G} = \Sl_2$, $\bm{H}$ is the diagonal torus, and $\bm{B}$ is the upper triangular matrices. Then $\bm{G}/\bm{H} \cong \mathbb{P}^1 \times \mathbb{P}^1 \setminus \Delta$, where $\Delta$ is the diagonal and $G$ acts on each component on the left by matrix multiplication. The homogeneous space $\mathbb{P}^1 \times \mathbb{P}^1 \setminus \Delta$ has divisor class group $\mathbb{Z}^2$. The palette $\bm{\mathcal{D}}$ consists of two colors: $V(x_{21})$ and $V(x_{22})$.
The associated characters are both defined by $(h_{ij}) \mapsto h_{22}$; call this character $\bm{\chi}$.
The vector space $\bm{\mathcal{N}}_\mathbb{Q}$ associated to $\bm{G}/\bm{H}$ is one-dimensional, $\bm{\mathcal{V}}$ is ray, and both colors lie outside of the valuation cone.

Now $G \cong \bm{G} \times (\mathbb{C}^*)^{\bm{\mathcal{D}}}$ and $H := \set{((h_{ij}),h_{22},h_{22}) : (h_{ij}) \in \bm{H}}$.
Then $G$ acts on $\mathbb{C}^2 \times \mathbb{C}^2 \times \mathbb{C}^*$.
In this action, $\bm{G}$ acts by left matrix multiplication on the copies of $\mathbb{C}^2$ and trivially on $\mathbb{C}^*$ and $(\mathbb{C}^*)^{\bm{\mathcal{D}}}$ respectively acts with weights $-\eta_1$ and $-\eta_2$ on the first and second copy of $\mathbb{C}^2$ and with weight $-\eta_1 - \eta_2$ on $\mathbb{C}^*$.
If the coordinates of the copies of $\mathbb{C}^2$ are respectively given by $S_{11}$, $S_{21}$ and $S_{12}$, $S_{22}$ and the coordinate of $\mathbb{C}^*$ is $T$, then 
\[
G/H = V(S_{11}S_{22} - S_{12}S_{21} - T) \subset \mathbb{C}^2 \times \mathbb{C}^2 \times \mathbb{C}^*.
\]
The map $\bm{\pi}: G/H \rightarrow \bm{G}/\bm{H}$ is then defined by 
\[
((S_{11},S_{21}),(S_{12},S_{22}),T) \mapsto [S_{11} : S_{21}] \times [S_{12} : S_{22}] \subset \mathbb{P}^1 \times \mathbb{P}^1 \setminus \Delta.
\]

Consider $\bm{Y} = \bm{G}/\bm{H}$; we will recover $\bm{\mathcal{V}}$ as $\trop_{\bm{G}}(\bm{G}/\bm{H})$ using Theorem \ref{V=G2}.
The colors of $G/H$ are given by $V(S_{21})$ and $V(S_{22})$ and we may set $f_1 = f_{11} := S_{21}$ and $f_2 = f_{21} := S_{22}$.
Further, $s_1 = s_2 = 2$ and $f_{12} = S_{11}$ and $f_{22} = S_{12}$, so the toric variety $Z_0$ associated to $G/H$ is $(\mathbb{C}^2 \setminus \set{0}) \times (\mathbb{C}^2 \setminus \set{0}) \times \mathbb{C}^*$ with these coordinates.
The pre-image $\bm{\pi}^{-1}(\bm{G}/\bm{H})$ is $G/H = V(S_{11}S_{22} - S_{12}S_{21} - T)$.
The tropicalization $\trop_\mathbb{T}(G/H)$ is the set of points
\[
(a_{11}, a_{12}, a_{21}, a_{22}, b_1) \in \trop_\mathbb{T}(Z_0) \cong \left(\overline{\mathbb{Q}}^2 \setminus \set{\infty} \right) \times \left( \overline{\mathbb{Q}}^2 \setminus \set{\infty} \right) \times \mathbb{Q}
\]
such that the minimum $\min\set{a_{11} + a_{22}, a_{12} + a_{21}, b_1}$
is met at least twice.
It follows that the image of $\trop_\mathbb{T}(G/H)$ under the map $\psi$ is the set of ordered triples $(a_1,a_2,b_1)$ such that $a_1 + a_2 \leq b_1$. 

The $\bm{B}$ semi-invariant rational functions on $\bm{G}/\bm{H}$ are one-dimensional, given by integer powers of $\bm{f} :=\frac{ S_{11}S_{22} - S_{12}S_{21}}{S_{12}S_{22}}$.
The lattice $\bm{\mathcal{N}}$ is spanned by the $\bm{G}$-invariant valuation $\bm{f} \mapsto 1$.
Under $\trop(\bm{\pi})$, an element $(a_1,a_2,b_1) \in \trop_G(G/H)$ is mapped to the valuation $(\bm{f} \mapsto b_1 - a_1 - a_2) \in \trop_{\bm{G}}(\bm{G}/\bm{H})$.
Because $a_1 + a_2 \leq b_1$ on $\psi\left( \trop_G(G/H) \right)$, it follows that this valuation is always non-negative on $\bm{f}$ and so we recover the fact that $\bm{\mathcal{V}} = \trop_{\bm{G}}(\bm{G}/\bm{H})$ is a ray in the positive direction.
\end{example}

\section{Embedding Toroidal Varieties in Toric Varieties}\label{toroidal}

In \cite{Ga}, Gagliardi considers spherical embeddings $X$ of $G/H$ whose associated colored fans only include rays without color. Given such an embedding, he finds an explicit toric variety $Z$ in which $G/H$ embeds such that $\overline{G/H} = X$ in $Z$. He restricts himself to such embeddings because to compute the Cox ring of a spherical embedding, only information about the $G$-stable prime divisors is needed, i.e. rays without color.

In this section we extend his construction to arbitrary toroidal spherical embeddings, whose colored fans contain no colors. We maintain the notation and conventions established in \textsection \ref{homspaces}. In particular, we start in the case where $G/H$ has trivial divisor class group. Let $X$ be a $G/H$-embedding given by a fan $\Sigma$ whose one-dimensional cones are $\set{u_1,\ldots, u_n} \subseteq \mathcal{V}$.
For each colored cone $(\sigma,\emptyset) \in \Sigma$, write $\sigma(1) \subseteq \set{u_1,\ldots, u_n}$ to denote the set of one-dimensional (non-colored) faces of a $\sigma$.
Define
\[
\mathfrak{A} := \set{\mathfrak{a} \subset \set{v_{ij}} : \text{ for each $i$ there is at least one $j$ with $v_{ij} \notin \mathfrak{a}$}}.
\]
We note that $\mathfrak{A}$ is defined slightly differently in \cite{Ga}, where the phrase ``at least one" is replaced by ``exactly one".
For each $\sigma$ and each $\mathfrak{a} \in \mathfrak{A}$, we define 
\[
\sigma_{\mathfrak{a}} := \text{cone}\left(\mathfrak{a} \cup \sigma(1) \right) \subset N_\mathbb{Q}.
\]
The dimension of the cone $\sigma_{\mathfrak{a}}$ is $\dim \sigma_{\mathfrak{a}} = |\mathfrak{a}| + \dim \sigma$.
This follows because for every $i$ at least one $v_{ij}$ is absent from $\mathfrak{a}$, which means the rays of $\sigma$ are linearly independent of the rays in $\mathfrak{a}$.
Then we define the fan $\Sigma_Z$ to be the set of cones $\sigma_{\mathfrak{a}}$ for all $(\sigma,\emptyset) \in \Sigma$ and $\mathfrak{a} \in \mathfrak{A}$:
\[
\Sigma_Z := \set{\sigma_{\mathfrak{a}} : (\sigma,\emptyset) \in \Sigma, \mathfrak{a} \in \mathfrak{A} }.
\]

That $\Sigma_Z$ is well-defined is a corollary of Proposition \ref{fanexists}, which we state and prove later.
Let $Z$ be the toric variety associated to the fan $\Sigma_Z$.
Note that $Z_0 \subseteq Z$ since the fan for $Z_0$ can be obtained from the definition by taking $\sigma$ to be the trivial cone and letting $\mathfrak{a}$ range over $\mathfrak{A}$.
We claim that the closure of $G/H$ in $Z$ is isomorphic to $X$. The argument mimics the procedure in \cite{Ga}; many of the proofs proceed essentially identically.
This claim is proved in Proposition \ref{embedding}; we delay the proof until \textsection \ref{badcolors}, where we simultaneously consider non-toroidal varieties.

We now show that the action of $G$ on $Z_0$ extends to $Z$. 
Define $\hat{N} := N \oplus \mathbb{Z}^n$ where $\mathbb{Z}^n$ is given the basis $\set{e_1,\ldots,e_n}$. 
For each $\sigma_{\mathfrak{a}} \in \Sigma_Z$, define the cone
\[
\hat{\sigma}_{\mathfrak{a}} := \text{cone}\left(\mathfrak{a} \cup \bigcup_{u_k \in \sigma(1)} \set{e_k} \right) \subseteq \hat{N}_\mathbb{Q}
\]
and let 
\[
\Sigma_{\hat{Z}} := \text{fan}\left( \hat{\sigma}_{\mathfrak{a}} : \sigma_{\mathfrak{a}} \in \Sigma_{Z} \right).
\]
In essence, the difference between $\sigma_\mathfrak{a}$ and $\hat{\sigma}_\mathfrak{a}$ is that the rays of $\sigma_\mathfrak{a}$ in $\sigma(1)$ are not necessarily orthogonal.
In $\hat{\sigma}_\mathfrak{a}$, each $u_k$ is replaced with a vector $e_k$ that is orthogonal to every vector in $\mathfrak{a}$ and every other $e_\ell$.

If $\hat{Z}$ is the associated toric variety with torus $\hat{\mathbb{T}} := \mathbb{T} \times (\mathbb{C}^*)^n$, then 
\[
\hat{Z} \cong \mathbb{C}^{s_1 + \cdots + s_r} \times (\mathbb{C}^*)^m \times \mathbb{C}^n \setminus \hat{S},
\]
where $\hat{S}$ is a closed set of codimension at least two.
We revisit this theory in more detail in \textsection \ref{extendedtrop}.
There is a natural toric morphism $p: \hat{Z} \rightarrow Z$ induced by the map $\hat{N} \rightarrow N$ defined by $v_{ij} \mapsto v_{ij}$, $w_k \mapsto w_k$, and $e_\ell \mapsto u_\ell$.
By Theorem 4.1 in \cite{Sw}, $p$ is a good quotient.
In fact, it is a geometric quotient (cf. \cite{AH} Proposition 3.2).
The quotient is with respect to a subtorus $\Gamma \cong (\mathbb{C}^*)^n$ of $\mathbb{T} \times \mathbb{C}^n$, which is parameterized as follows: 
\[
\Gamma := \set{ \left( \prod_{\ell = 1}^n t_\ell^{-\langle u_\ell, v_{11}^* \rangle}, \ldots, \prod_{\ell = 1}^n t_\ell^{-\langle u_\ell, v_{rs_r}^* \rangle},\prod_{\ell = 1}^n t_\ell^{-\langle u_\ell, w_1^* \rangle}, \ldots, \prod_{\ell = 1}^n t_\ell^{-\langle u_\ell, w_k^* \rangle}, t_1, \ldots, t_n \right) : t_\ell \in \mathbb{C}^*}.
\]

We may extend the natural action of $G$ on $\mathbb{C}^{s_1 + \cdots + s_r} \times (\mathbb{C}^*)^m$ to $\hat{Z}$ by having $G$ act trivially on the additional $\mathbb{C}^n$ summand.
Then the action of $G$ on $\hat{Z}$ commutes with the action of $\Gamma$.
As a result, we obtain an action of $G$ on $Z$.
More explicitly, because $p$ is a good quotient, Proposition 5.0.7 of \cite{CLS} says that for any point $x \in X$, the preimage $p^{-1}(x)$ contains a unique closed $\Gamma$-orbit.
If $g \in G$ and $z \in Z$, we define $g \cdot z := p(g \cdot \hat{z})$, where $\hat{z}$ is an element of the unique closed $\Gamma$-orbit in $p^{-1}(z)$.
Because the actions of $G$ and $\Gamma$ commute, $g$ takes the closed orbit of $p^{-1}(z)$ to another $\Gamma$-orbit.
Since $p$ is constant on orbits, $p(g \cdot \hat{z})$ does not depend on the choice of $\hat{z}$.

Further, the inclusion $N \hookrightarrow \hat{N}$ induces a morphism $Z_0 \rightarrow \hat{Z}$ that commutes with $p: \hat{Z} \rightarrow Z$ to give the inclusion $Z_0 \hookrightarrow Z$, so the described action of $G$ on $Z$ extends that of $G$ on $Z_0$.

\begin{example}
We again consider Example \ref{punctured affine in toric} and in particular the embedding of $\mathbb{C}^2 \setminus \set{0}$ in $\mathbb{P}^2 \setminus \set{0}$.
The corresponding colored cone has one ray $u_1$ spanned by $-v_1$ in $\mathcal{N}_\mathbb{Q}$ and $Z = \mathbb{P}^2 \setminus \set{0}$.
The variety $\hat{Z} = \left( \mathbb{C}^2 \times \mathbb{C} \right) \setminus \left(\set{0} \times \mathbb{C}\right) = \left(\mathbb{C}^2 \setminus \set{0}\right) \times \mathbb{C}$ is given by adding one $\mathbb{C}$ summand for the cone $-v_1$ and removing the subvariety where both of the first two coordinates vanish.
The morphism $p: \hat{Z} \rightarrow Z = \mathbb{P}^2 \setminus \set{0}$ is then given by $(x,y,z) \mapsto [x : y : z]$.
In Figure \ref{lattice map}, we illustrate the compatible map of fans between $\Sigma_{\hat{Z}}$ and $\Sigma_Z$ that induces $p$.
The preimage of a point $[a : b : c] \in \mathbb{P}^2 \setminus \set{0}$ is $\set{(at,bt,ct) : a,b,c \in \mathbb{C}, t \in \mathbb{C}^*}$, so $p$ is a geometric quotient by the action of the torus $\Gamma := \set{(t,t,t) : t \in \mathbb{C}^*}$. 

\begin{center}
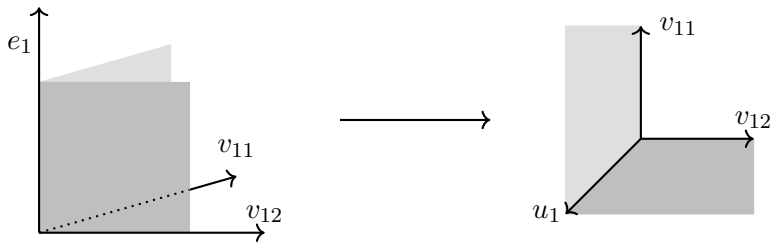
\begin{figure}[H]
\begin{tikzpicture}
\draw[rectangle,gray!25, fill = gray!25] (0,0)--(0,2)--(1.75,2.5)--(1.75,.5);
\draw[thick,->] (0,0)--(2.625,.75);
\draw[rectangle,gray!50, fill = gray!50] (0,0)--(0,2)--(2,2)--(2,0);
\draw[thick,->] (0,0)--(0,3);
\draw[thick,->] (0,0)--(3,0);
\draw[thick,dotted] (0,0)--(2,.57142857);

\draw[thick,->] (4,1.5)--(6,1.5);

\draw[rectangle,gray!50,fill = gray!50] (8,1.25)--(9.5,1.25)--(9.5,.25)--(7,.25);
\draw[rectangle,gray!25,fill = gray!25] (8,1.25)--(8,2.75)--(7,2.75)--(7,.25);
\draw[thick,->] (8,1.25)--(8,2.75);
\draw[thick,->] (8,1.25)--(9.5,1.25);
\draw[thick,->] (8,1.25)--(7,.25);

\draw (-.25,2.5) node[]{$e_1$};
\draw (3,.25) node[]{$v_{12}$};
\draw (2.625,1.125) node[]{$v_{11}$};

\draw (6.75,.25) node[]{$u_1$};
\draw (8.5,2.75) node[]{$v_{11}$};
\draw (9.5,1.5) node[]{$v_{12}$};
\end{tikzpicture}
\caption{The compatible map of fans induced by $\hat{N}_\mathbb{Q} \rightarrow N_\mathbb{Q}$ for the embedding $\mathbb{P}^2 \setminus \set{0}$.}
\label{lattice map}
\end{figure}
\end{center}

If instead we consider $\Bl_0(\mathbb{C}^2) \subset \mathbb{C}^2 \times \mathbb{P}^1$ with one ray spanned by $v_1$, we get the same $\hat{Z}$ and the map $\hat{Z} \rightarrow Z = \Bl_0(\mathbb{C}^2)$ is given by $(x,y,z) \mapsto (xz,yz) \times [x : y]$.
The preimage of a point $(ac,bc) \times [a : b] \in \Bl_0(\mathbb{C}^2)$ is $\set{(at^{-1},bt^{-1},ct) : t \in \mathbb{C}^*}$, so it is a geometric quotient by the torus $\Gamma := \set{(t^{-1},t^{-1},t) : t \in \mathbb{C}^*}$.
\end{example}

We have not yet discussed the case of a $\bm{G}/\bm{H}$-embedding $\bm{X}$ where the homogeneous space $\bm{G}/\bm{H}$ has nontrivial divisor class group.
Given such a homogeneous space, we earlier defined an associated homogeneous space $G/H$ with trivial divisor class group along with a map $\bm{\pi}: G/H \rightarrow \bm{G}/\bm{H}$. Gagliardi shows (\cite{Ga}, Proposition 3.4) that the pushforward $\bm{\pi}_*: \mathcal{N}_\mathbb{Q} \rightarrow \bm{\mathcal{N}}_\mathbb{Q}$ is an isomorphism when restricted to the subspace $(\mathcal{N}_T)_\mathbb{Q} := \thespan\set{w_1,\ldots,w_m}$.

If $\bm{\Sigma}$ is the colored fan associated to a toroidal embedding $\bm{X}$, we define a $G/H$-embedding $X$ by considering the fan $\Sigma$, which is the preimage of $\bm{\Sigma}$ under $\bm{\pi}_*|_{\left(\mathcal{N}_T\right)_\mathbb{Q}}$.
Let $Z$ be the toric variety associated to $X$ as constructed earlier in this section.

There is a good geometric quotient of $Z$ by the torus $\left(\mathbb{C}^*\right)^{\bm{\mathcal{D}}}$, so we may extend $\bm{\pi}$ to $\bm{\pi}: Z \rightarrow \bm{Z}$, where $\bm{Z}$ is a toric variety.
Refer again to the theory in \cite{Sw}, as well as \cite{AH}, Proposition 3.2.
By Proposition \ref{embedding}, the spherical variety $X$ associated to the fan $\Sigma$ is the closure of $G/H$ in $Z$.
Then $\bm{\pi}$ takes $X$ to the closure of $\bm{G}/\bm{H}$ in $\bm{Z}$ (cf. Definition 5.0.5 and Theorem 5.0.6 of \cite{CLS}). 
By the construction of the colored fan, the image of $X$ under $\bm{\pi}$ is $\bm{X}$ and hence $\bm{X}$ is the closure of $\bm{G}/\bm{H}$ in $\bm{Z}$.
We collect these facts later in Proposition \ref{nontriv} where we also consider the possibility of colors.

\section{The Problem of Colors}\label{badcolors}

The methods of \textsection \ref{toroidal} are illustrations of Proposition 2.11 of \cite{HK}, which says that any Mori dream space can be embedded in a projective toric variety. Projective spherical varieties are Mori dream spaces, but not every spherical embedding is a Mori dream space. This issue is addressed in \cite{Ga2}, another paper of Gagliardi.

\begin{definition}\label{A2}
A normal variety $X$ is said to have the \emph{$A_2$-property} if any two points in $X$ lie in some affine open subset of $X$.
\end{definition}

\begin{definition}
A colored fan $\Sigma$ is \emph{polyhedral} if the relative interiors of any two colored cones in $\Sigma$ have empty intersection. 
\end{definition}

\begin{theorem}[\cite{Ga2}, Theorem 1.5]\label{A2poly} Let $\Sigma$ be a colored fan with associated spherical embedding $G/H \hookrightarrow X$. Then $X$ has the $A_2$-property if and only if $\Sigma$ is polyhedral.
\end{theorem}

Finally, we make use of the following theorem of W\l odarczyk, which Gagliardi cites in \cite{Ga2}:

\begin{theorem}[\cite{Wl}]
A normal variety $X$ has the $A_2$-property if and only if it admits a closed embedding $X \hookrightarrow Z$ into a toric variety $Z$.
\end{theorem}

All of this serves to tell us that the spherical varieties that cannot be embedded in a toric variety are those whose cones are not polyhedral. Intuitively, this is reflected by the fact that a spherical embedding with a non-polyhedral colored fan will be associated by the process described in \textsection \ref{toroidal} to a toric variety with a non-polyhedral fan, which is an impossibility.
A colored fan can only exhibit non-polyhedral behavior outside of the valuation cone.
We illustrate this in the following extended example.

\begin{example}\label{colors_outside_val}
Let $G = \Sl_3$ and $H = \Sl_{2}$ embedded in $G$ as the lower right entries.
Then $G$ has an action on $\mathbb{C}^3 \times \mathbb{C}^3$ given by 
\[
g \cdot (x,y) = \left( gx, \left( g^{-1} \right)^*y \right),
\]
where $*$ indicates taking the conjugate transpose.
Under this action, the point $((1,0,0),(1,0,0))$ has isotropy group $H$.
If the coordinates of $\mathbb{C}^3 \times \mathbb{C}^3$ are given as $((x_1,x_2,x_3),(y_1,y_2,y_3))$, then the orbit of this point is $V(x_1y_1 + x_2y_2 + x_3y_3 - 1) = G/H$.
Taking the Borel group $B$ consisting of the upper triangular matrices, $G/H$ is a spherical homogeneous space.

There are two colors: $V(x_3)$ and $V(y_1)$, and $\Gamma(G/H,\mathcal{O}_{G/H}^*)$ is trivial.
Thus, $\mathcal{N}$ is two-dimensional, spanned by valuations $v_1$ and $v_2$ respectively associated to the colors $V(x_3)$ and $V(y_1)$.
The $G$-modules $G\cdot x_3$ and $G \cdot y_1$ both have rank three and we may define an embedding as follows:
\begin{align*}
V(x_1y_1 + x_2y_2 + x_3y_3 - 1) = G/H & \hookrightarrow Z_0 := \left( \mathbb{C}^3 \setminus \set{0} \right) \times \left(\mathbb{C}^3 \setminus \set{0} \right) \\
((x_1,x_2,x_3),(y_1,y_2,y_3)) & \mapsto (x_3,x_2,x_1,y_1,y_2,y_3)
\end{align*}
The tropicalization $\trop_\mathbb{T}(G/H)$ of $G/H$ in $Z_0$ is the set of extended valuations $\mu$ where the minimum
\[
\min\set{\mu(x_1) + \mu(y_1),\mu(x_2) + \mu(y_2),\mu(x_3) + \mu(y_3),0}
\]
is met at least twice.
By Theorem \ref{GagVal}, $\mathcal{V} = \trop_\mathbb{T}(G/H) \cap \mathcal{N}_\mathbb{Q}$.
On $\mathcal{N}_\mathbb{Q}$, $\mu(x_1) = \mu(x_2) = \mu(x_3)$ and $\mu(y_1) = \mu(y_2) = \mu(y_3)$, and so for the above minimum to be met twice, we must have $\mu(x_3) + \mu(y_1) \leq 0$. Thus, $\mathcal{V} = \set{v_1 + v_2 \leq 0}$. Figure \ref{valcone} shows the valuation cone and palette of this homogeneous space.
\begin{center}
\begin{figure}[H]
\begin{tikzpicture}
\draw[fill,gray!25] (2,-2)--(-2,-2)--(-2,2);
\draw[thick,->] (0,0)--(2,-2);
\draw[fill] (0,0) circle(.05);
\draw[thick,->] (0,0)--(-2,2);

\draw[->,dotted] (0,0)--(2,0);
\draw[->,dotted] (0,0)--(-2,0);
\draw[->,dotted] (0,0)--(0,2);
\draw[->,dotted] (0,0)--(0,-2);

\draw[red,fill=red] (1,0) circle(.05);
\draw[red] (1,0) circle(.1);

\draw[blue,fill=blue] (0,1) circle(.05);
\draw[blue] (0,1) circle(.1);

\draw (2.3,0) node {$v_1$};
\draw (0,2.15) node {$v_2$};

\draw (1.3,.35) node {$\rho(V(x_3))$};
\draw (.85,1.35) node {$\rho(V(y_1))$};
\end{tikzpicture}
\caption{The valuation cone and palette of $V(x_1y_1 + x_2y_2 + x_3y_3 - 1)$.}
\label{valcone}
\end{figure}
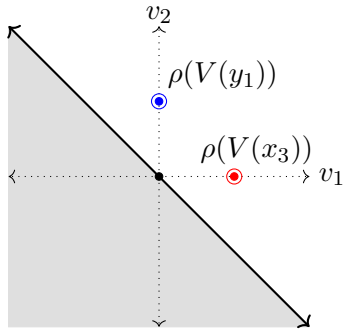
\end{center}
The interest here is that there are two colors that lie outside the valuation cone, so there is potential for valid colored cones that are not polyhedral.
Indeed, we may have a simple $G/H$-embedding corresponding to the colored cone spanned by the color $\rho(V(x_3))$ and $-2v_1 + v_2$ since the interior of this cone intersects the valuation cone. This is illustrated by the red cone in Figure \ref{redbluepurplecone}. Similarly, there exists an embedding whose colored cone is spanned by $\rho(V(y_1))$ and $v_1 - 2v_2$, the blue cone in Figure \ref{redbluepurplecone}.
The union of these cones is a valid colored fan and thus corresponds to a $G/H$-embedding.
Even though this fan is not polyhedral, it is a colored fan because the intersection is outside the valuation cone.
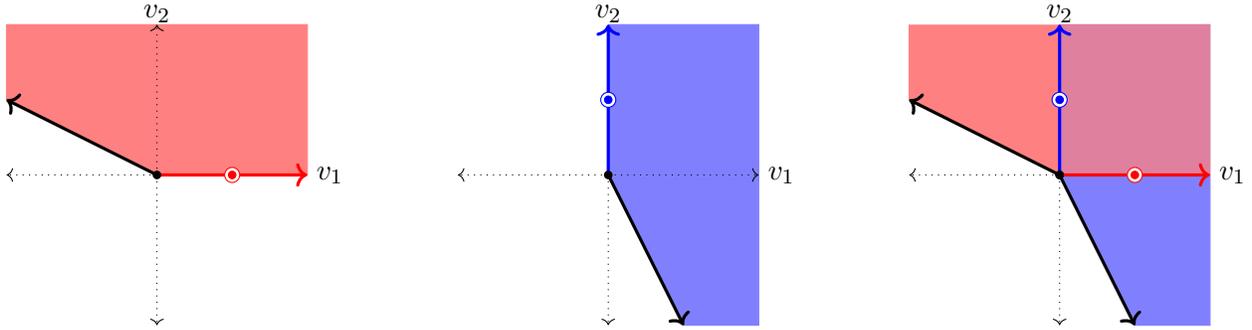
\begin{figure}[h]
\begin{center}
\begin{tikzpicture}


\draw[fill,red!50] (2,0)--(2,2)--(-2,2)--(-2,1)--(0,0);
\draw[very thick,red,->] (0,0)--(2,0);
\draw[fill] (0,0) circle(.05);
\draw[very thick,->] (0,0)--(-2,1);

\draw[->,dotted] (0,0)--(-2,0);
\draw[->,dotted] (0,0)--(0,2);
\draw[->,dotted] (0,0)--(0,-2);

\draw[red, fill= white] (1,0) circle(.1);
\draw[red,fill=red] (1,0) circle(.05);

\draw (2.3,0) node {$v_1$};
\draw (0,2.15) node {$v_2$};


\draw[fill,blue!50] (6,2)--(8,2)--(8,-2)--(7,-2)--(6,0);
\draw[very thick,blue,->] (6,0)--(6,2);
\draw[fill] (6,0) circle(.05);
\draw[very thick,->] (6,0)--(7,-2);

\draw[->,dotted] (6,0)--(8,0);
\draw[->,dotted] (6,0)--(4,0);
\draw[->,dotted] (6,0)--(6,-2);

\draw[blue, fill= white] (6,1) circle(.1);
\draw[blue,fill=blue] (6,1) circle(.05);

\draw (8.3,0) node {$v_1$};
\draw (6,2.15) node {$v_2$};


\draw[fill,blue!50] (14,0)--(14,-2)--(13,-2)--(12,0);
\draw[fill,red!50] (10,2)--(10,1)--(12,0)--(12,2);
\draw[fill,purple!50] (12,0)--(14,0)--(14,2)--(12,2);

\draw[very thick,blue,->] (12,0)--(12,2);
\draw[fill] (12,0) circle(.05);
\draw[very thick,->] (12,0)--(13,-2);
\draw[very thick,red,->] (12,0)--(14,0);
\draw[fill] (12,0) circle(.05);
\draw[very thick,->] (12,0)--(10,1);

\draw[->,dotted] (12,0)--(10,0);
\draw[->,dotted] (12,0)--(12,-2);

\draw[blue, fill= white] (12,1) circle(.1);
\draw[blue,fill=blue] (12,1) circle(.05);

\draw[red, fill= white] (13,0) circle(.1);
\draw[red,fill=red] (13,0) circle(.05);

\draw (14.3,0) node {$v_1$};
\draw (12,2.15) node {$v_2$};
\end{tikzpicture}
\caption{Colored cones corresponding to three spherical embeddings. The embedding on the far right is the gluing of the other two along $G/H$; it has two maximal colored cones that overlap in the first quadrant.}
\label{redbluepurplecone}
\end{center}
\end{figure}
\end{example}

Throughout this section we therefore only consider colored fans that are polyhedral.
As before, we assume for the moment that $G/H$ has trivial divisor class group.
Let $X$ be a $G/H$-embedding with colored fan $\Sigma$ and let $Z_0$ be the toric variety associated to $G/H$ defined in \textsection \ref{toroidal}.
Let $(\sigma,\mathcal{F}) \in \Sigma$ be a colored cone. Then write $\sigma(1) := \set{u_1,\ldots,u_n}$ for the set of one-dimensional non-colored faces of $\sigma$ and define
\[
\mathfrak{A}(\mathcal{F}) := \set{\mathfrak{a} \subseteq \set{v_{ij}} : \text{for each } i \text{ such that } D_i \notin \mathcal{F}, \text{ there is at least one $j$ with } v_{ij} \notin \mathfrak{a} }.
\]
Note that $\mathfrak{A}(\emptyset)$ is the same as $\mathfrak{A}$, as defined at the beginning of \textsection \ref{toroidal}.
The set $\mathfrak{A}(\mathcal{F})$ simply extends $\mathfrak{A}$ by allowing the entire set $\set{v_{ij}}_j$ to be present when the color $D_i$ lies in $\mathcal{F}$.
For a given $\mathfrak{a} \in \mathfrak{A}(\mathcal{F})$, we define a cone in $N_\mathbb{Q}$ as follows:
\[
\sigma_{\mathfrak{a}} := \text{cone}\left(\mathfrak{a} \cup \sigma(1) \right) \subset N_\mathbb{Q}.
\]
The fan $\Sigma_Z$ is defined similarly to before:
\[
\Sigma_Z := \text{fan}\left( \set{\sigma_{\mathfrak{a}} : (\sigma,\mathcal{F}) \in \Sigma, \mathfrak{a} \in \mathfrak{A}(\mathcal{F})} \right).
\]
If $\Sigma$ has no colors, then $\Sigma_Z$ is consistent with the object we described in \textsection \ref{toroidal}.
By a similar argument to the toroidal case, we can see that the action of $G$ on $Z_0$ also extends to $Z$ in this setting.

\begin{proposition}\label{fanexists}
If $X$ is a $G/H$-embedding whose associated colored fan $\Sigma$ is polyhedral, then the fan $\Sigma_Z$ is well-defined.
\end{proposition}
\begin{proof}
We must check that $\Sigma_Z$ is closed under taking faces and that the intersection of any two cones in $\Sigma_Z$ is a face of each.
If $\sigma_{\mathfrak{a}}$ is a cone in $\Sigma_Z$, then its faces are precisely those $\sigma'_{\mathfrak{a}'}$ such that $\mathfrak{a}' \subseteq \mathfrak{a}$ and $\sigma' \preceq \sigma$ in $\Sigma$.
Such a cone $\sigma'_{\mathfrak{a}'}$ is in $\Sigma_Z$, as needed.
Suppose $\sigma_{\mathfrak{a}}, \sigma'_{\mathfrak{a}'} \in \Sigma_Z$ are two cones.
Then because $\Sigma$ is a polyhedral fan, the intersection of $\sigma$ and $\sigma'$ is a colored cone $(\sigma'',\mathcal{F}'') \in \Sigma$ that is a colored face of each.
The intersection $\sigma_{\mathfrak{a}} \cap \sigma'_{\mathfrak{a}'}$ then equals $\sigma''_{\mathfrak{a} \cap \mathfrak{a}'}$, which is in $\Sigma_Z$.
\end{proof}

The following proposition is what this section has been working towards. It is an extension of Lemma 2.13 and Proposition 2.14 in \cite{Ga}, and the proof owes much of its structure to those two results.


\begin{proposition}[cf. \cite{Ga}, Proposition 2.14]\label{embedding}
Suppose $G/H$ has trivial divisor class group and $X$ is a $G/H$-embedding corresponding to a colored fan $\Sigma$. Then $X$ is isomorphic to the closure $\overline{G/H}$ of $G/H$ in the associated toric variety $Z$.
\end{proposition}
\begin{proof}

Proposition 2.14 of \cite{Ga} proves the claim when $\Sigma$ consists solely of non-colored rays.
If $\Sigma$ consists of a single ray $\sigma$ spanned by the color $D_i$, then on every additional $G$-orbit of $Z$ outside of $Z_0$, $f_i$ vanishes and $f_j$ and $g_k$ do not for every $j \neq i$ and $k$.
Thus, this embedding corresponds to a ray in the direction of $\sigma$.
Moreover, this ray must have color since the $f_{ij}$ may vanish with different multiplicities along the orbits added.
After gluing together along shared orbits, we conclude that the statement of the proposition holds when $\Sigma$ consists solely of rays, both colored and not.

We now turn to proving the full claim.
It will be sufficient to prove it for simple embeddings since they can then be glued together along shared orbits.
Let $\Sigma$ consist of a single colored cone $(\sigma,\mathcal{F})$ with non-colored rays $u_1,\ldots,u_n \in \sigma(1)$.

Let the maximal cone in $\Sigma_Z$ spanned by the sets $\sigma(1)$ and $\bigcup_{D_i \in \mathcal{F}} \set{v_{ij}}_{j=1}^{s_i}$ be denoted $\tau$. 
The cone $\tau$ corresponds to an affine variety $U_\tau$.
For each ray $u_\ell \in \sigma(1)$, there is an open affine subset $U_\ell$ of $U_\tau$ with two torus orbits.
Similarly, for each $D_i \in \mathcal{F}$ and each $1 \leq j \leq s_i$, there is an open affine subset $U_{ij}$ corresponding to the ray $v_{ij}$.
Let $\pi_\ell \in \mathbb{C}[U_\tau]$ for $1 \leq \ell \leq n$ denote prime elements that cut out the closures of the $U_\ell$ in $U_\tau$.
Similarly choose prime elements $\pi_{ij} \in \mathbb{C}[U_\tau]$ that cut out the $U_{ij}$.
Then we have the following commutative diagram:
\begin{center}
\begin{tikzcd}
\mathbb{C}[U_\tau] \ar{d} \arrow[hookrightarrow]{rrr} &&& (\mathbb{C}[S_{11},\ldots, S_{rs_r},T_1,\ldots,T_m])_\mathfrak{p} \ar{d} \\
\mathbb{C}[U_\tau]/(\mathfrak{p} \cap \mathbb{C}[U_\tau]) \ar[hookrightarrow]{r}{\mathfrak{R}} & R_{1} \ar[hookrightarrow]{r}{\mathfrak{L}} & R_{2} \ar{r} & (\mathbb{C}[S_{11},\ldots, S_{rs_r},T_1,\ldots,T_m])_\mathfrak{p}/\mathfrak{p} \cong \mathbb{C}(G/H)
\end{tikzcd}
\end{center}
In the diagram, $\mathfrak{R}$ is normalization and $\mathfrak{L}$ is localization.
For each $\pi_\ell$ and each $\pi_{ij}$, there are respectively prime elements $\tilde{\pi}_\ell,\tilde{\pi}_{ij} \in R_2$
such that $V(\pi_\ell) = V(\tilde{\pi}_\ell)$ and $V(\pi_{ij}) = V(\tilde{\pi}_{ij})$ in $\spec{R_2}$.
Now $L \in \set{S_{ij}, T_k}$ can be written in the form 
\[
L = c \cdot \prod_{\ell=1}^n \left(\pi_\ell^{d_{\ell,1}}/\pi_\ell^{d_{\ell,2}} \right) \cdot \prod_{D_i \in \mathcal{F}, 1 \leq j \leq s_i} \left(\pi_{ij}^{d_{ij,1}}/\pi_{ij}^{d_{ij,2}} \right)
\] 
for $d_\ell,d_{i,j} \in \mathbb{Z}_{\geq 0}$ and $c \in \mathbb{C}[U_\tau]^*$.
We may further write
\[
L = \tilde{c} \cdot \prod_{\ell=1}^n \left(\tilde{\pi}_\ell^{e_\ell d_{\ell,1}}/\tilde{\pi}_\ell^{e_\ell d_{\ell,2}} \right) \cdot \prod_{D_i \in \mathcal{F}, 1 \leq j \leq s_i} \left(\tilde{\pi}_{ij}^{e_{ij}d_{ij,1}}/\tilde{\pi}_{ij}^{e_{ij}d_{ij,2}} \right)
\] 
where $e_\ell,e_{ij} \in \mathbb{Z}_{\geq 0}$ and $\tilde{c} \in R_{2}^*$.

The valuations in the colored cone $(\sigma,\mathcal{F})$ are positive $\mathbb{Q}$-linear combinations of the $u_\ell$ and the valuations induced by the colors $D_i \in \mathcal{F}$.
It follows from the above argument that every such valuation is induced by a torus invariant one in $Z$.
This proves the claim for simple spherical embeddings. Gluing together along shared orbits gives the full result.
\end{proof}

It still remains to consider when the homogeneous space $\bm{G}/\bm{H}$ has nontrivial divisor class group.
Recall that the colors $D_1,\ldots,D_r$ of the associated homogeneous space $G/H$ are precisely the pullbacks of the colors $\bm{D}_1,\ldots,\bm{D}_r$ of $\bm{G}/\bm{H}$.
For each $D_i$, Gagliardi defines $f_i \in \Gamma(G/H,\mathcal{O}_{G/H})$ such that $V(f_i) = D_i$.
There is an inclusion from the character lattice of $(\mathbb{C}^*)^{\bm{\mathcal{D}}}$ to $\mathbb{C}[G]^*$ that takes a character $\chi$ to a monomial $\epsilon^\chi \in \mathbb{C}\left[ \left( \mathbb{C}^* \right)^{\bm{\mathcal{D}}} \right] \subset \mathbb{C}[G]^*$.
Then we may write $f_i := \bm{f}_i\epsilon^{-\eta_i}$, where $\eta_i$ is the character that is trivial on every coordinate not equal to $i$.
The function $f_i$ is invariant under the action of $H$ from the right and $V(f_i) = D_i$.
Note that as a result $(\mathbb{C}^*)^{\bm{\mathcal{D}}}$ has a nontrivial action on $f_i$ for all $i$.

Then let $\bm{X}$ be a $\bm{G}/\bm{H}$-embedding with polyhedral colored fan $\bm{\Sigma}$.
For each color $\bm{\mathcal{D}_i}$ appearing in $\bm{\Sigma}$, include in $\Sigma$ a colored ray corresponding to $D_i$.
For each ray without color, take the preimage of this ray under $\bm{\pi}_*|_{{(\mathcal{N}_T)}_\mathbb{Q}}$ as we did for toroidal embeddings.
Then add in higher-dimensional cones between these rays in $\mathcal{N}_\mathbb{Q}$ if they exist in $\bm{\Sigma}$.
As in the toroidal case, there is a good geometric quotient $\bm{\pi}: Z \rightarrow \bm{Z}$ that extends $\bm{\pi}: G/H \rightarrow \bm{G}/\bm{H}$.
We obtain the following result, similarly to the toroidal case:
\begin{proposition}\label{nontriv}
Let $\bm{G}/\bm{H} \hookrightarrow \bm{X}$ be a spherical embedding with associated spherical embedding $G/H \hookrightarrow X$ where $G/H$ has trivial divisor class group.
Then $\overline{\bm{G}/\bm{H}} \cong \bm{X}$, where the closure is taken in $\bm{Z}$.
\end{proposition} 

%

\section{Extended (Global) Tropicalization}\label{extendedtrop}

This section finally gives a global construction for the tropicalization of a spherical embedding.
Because of the issues raised in \textsection \ref{badcolors}, this construction can only work for a certain class of spherical embeddings, namely those that have the $A_2$-property.
For many homogeneous spaces, this consideration will not raise any issues.
For example, if $G/H$ is horospherical or if it has fewer than two colors lying outside the valuation cone, the global tropicalization will always go through without issue.

As before, we begin in the case that $G/H$ has trivial divisor class group and we maintain the notation introduced previously.
Let $X$ be a $G/H$-embedding with the associated toric varieties $Z$ and $\hat{Z}$ and morphism $p: \hat{Z} \rightarrow Z$. 
Our methodology will be to work in the toric world with $\hat{Z}$ and $Z$ where results are known. Then we will apply an extension of the map $\psi$.
The general theory in the following discussion can be found in more generality in \textsection 6.1 of \cite{MS}; also refer to \cite{CLS}, \textsection 5.1.

There is a short exact sequence
\[
0 \rightarrow N \hookrightarrow \hat{N} \rightarrow A_{n-1}\left(\hat{Z}\right) \rightarrow 0,
\]
where $A_{n-1}\left(\hat{Z} \right)$ is the cokernel of the natural inclusion $N \hookrightarrow \hat{N}$.
Applying the functor $\Hom(-,\mathbb{C}^*)$, we obtain the following exact sequence:
\[
\Hom(N,\mathbb{C}^*) \leftarrow \Hom\left(\hat{N},\mathbb{C}^* \right) \hookleftarrow Q \leftarrow 0,
\]
where $Q := \Hom\left(A_{n-1}\left(\hat{Z}\right),\mathbb{C}^*\right)$.
Let $E_\ell$ denote the coordinate in $\mathbb{C}\left[\hat{Z}\right]$ corresponding to the ray $e_\ell$ so that 
\[
\mathbb{C}\left[\hat{Z}\right] = \mathbb{C}[S_{11},\ldots,S_{rs_r},T_1,\ldots,T_m,E_1,\ldots,E_n].
\]
Then the \emph{irrelevant ideal} in $\mathbb{C}\left[\hat{Z} \right]$ is 
\[
F := \left\langle \prod_{v_{ij} \notin \sigma} S_{ij} \cdot \prod_{w_{k} \notin \sigma} T_k \cdot \prod_{e_\ell \notin \sigma} E_\ell : \sigma \in \Sigma_{\hat{Z}}\right\rangle
\]
and 
\[
\hat{Z} \cong \mathbb{C}^{s_1 + \cdots + s_r + m + n} \setminus V(F) \cong \mathbb{C}^{s_1 + \cdots + s_r} \times (\mathbb{C}^*)^m \times \mathbb{C}^n \setminus \hat{S}
\]
for some $\hat{S}$ of codimension at least two.
Finally, we have that 
\[
Z \cong (\mathbb{C}^{s_1 + \cdots + s_r + m + n} \setminus V(F))/Q.
\]
This quotient construction of $Z$ tropicalizes in the following sense (see \cite{MS}, Proposition 6.2.6 and Corollary 6.2.16):
\begin{proposition}
Suppose $G/H$ has trivial divisor class group and suppose $Y \subseteq G/H$ is a closed subvariety. Let $X$ be a $G/H$-embedding with the $A_2$-property with associated toric varieties $Z$ and $\hat{Z}$ and let $\overline{Y}$ be the closure of $Y$ in $Z$.
Then 
\[
\trop_{\mathbb{T}}\left(\overline{Y} \right) \cong \left( \overline{\mathbb{Q}}^{s_1 + \cdots + s_r + m + n} \setminus \trop_{\hat{\mathbb{T}}}(V(F)) \right)/\trop_{\hat{\mathbb{T}}}(Q).
\]
\end{proposition}

This result provides a means for globally tropicalizing a toric variety. That is to say, we do not need to consider any separate pieces and then glue together, we may simply tropicalize a single toric variety and take an appropriate quotient.
To obtain a universal tropicalization for spherical embeddings, it will be sufficient to describe how to recover $\trop_G\left( \overline{Y} \right)$ from $\trop_\mathbb{T}\left( \overline{Y} \right)$.
When $\overline{Y}$ is replaced by $Y$, we described this in Theorem \ref{V=G} using the piecewise projection map $\psi$:
\begin{align*}
\psi: \trop_{\mathbb{T}}(Z_0) & \rightarrow \mathcal{N}_\mathbb{Q} \\
(a_{11},\ldots,a_{1s_1},a_{21},\ldots,a_{rs_r},b_1,\ldots,b_m) & \mapsto \left(\min_{1 \leq j \leq s_1} \set{a_{1j}},\ldots,\min_{1 \leq j \leq s_r}\set{a_{rj}},b_1,\ldots,b_m \right).
\end{align*}

We now define an extension $\overline{\psi}: \trop_{\mathbb{T}}(Z) \rightarrow \trop_G(X)$ of $\psi$ that takes extended $\mathbb{T}$-invariant valuations to extended $G$-invariant valuations.
Applying this map to the global toric tropicalization above will afford a global spherical tropicalization.
Let $(\sigma,\mathcal{F})$ be a colored cone in the colored fan of the $G/H$-embedding $X$ and let $\mathfrak{a} \in \mathfrak{A}(\mathcal{F})$.
In the extended tropicalization of the spherical variety, tropicalizing the orbit corresponding to the colored cone $(\sigma,\mathcal{F})$ corresponds to adding in semigroup homomorphisms in $\Hom^{\mathcal{V}}{\left(\sigma^\vee \cap \mathcal{M},\overline{\mathbb{Q}}\right)}$ and viewing them as limit points of $\trop_G(G/H) := \mathcal{V} \subseteq \Hom(\mathcal{M}, \mathbb{Q})$.

Before proceeding, we deal with a slight clash of notation that comes up here.
Thought of as a valuation, $\mu \in \Hom{(\sigma_\mathfrak{a}^\vee \cap \mathcal{M}, \overline{\mathbb{Q}})}$ acts on a lattice of torus semi-invariant rational functions, where the group action is multiplicative.
As an element of $\Hom{(\sigma_\mathfrak{a}^\vee \cap \mathcal{M}, \overline{\mathbb{Q}})}$, however, $\mu$ acts on the additive semigroup $\sigma_\mathfrak{a}^\vee \cap M$.
This can be addressed this by identifying an element $(a_{11},\ldots,a_{rs_r},b_1,\ldots,b_m) \in \sigma_\mathfrak{a}^\vee \cap M$ with the function $f_{11}^{a_{11}} \cdots f_{rs_r}^{a_{rs_r}}g_1^{b_1} \cdots g_m^{b_m}$.
Similarly, $(a_1,\ldots,a_r,b_1,\ldots,b_m) \in \sigma^\vee \cap \mathcal{M}$ is identified with $f_1^{a_1} \cdots f_r^{s_r} g_1^{b_1} \cdots g_m^{b_m}$.

The extension $\overline{\psi}$ will take an extended valuation $\mu \in \Hom{(\sigma_\mathfrak{a}^\vee \cap M, \overline{\mathbb{Q}})}$ to $\Hom^{\mathcal{V}}{(\sigma^\vee \cap \mathcal{M}, \overline{\mathbb{Q}})}$.
Suppose $\mu \in \trop_\mathbb{T}(O)$ where $O$ corresponds to the cone $\sigma_\mathfrak{a}$.
Further let $(\sigma,\mathcal{F})$ be the associated colored cone corresponding to a $G$-orbit $\mathcal{O}$.
Define a set $\Omega \subset \mathbb{Z}^{r}$ as follows:
\[
\Omega := \set{\omega \in \mathbb{Z}^r : 1 \leq \omega(i) \leq s_i \text{ for all } 1 \leq i \leq r}.
\]

Then for any $\mu \in \Hom\left(\sigma_\mathfrak{a}^\vee \cap M, \overline{\mathbb{Q}} \right)$, we define $\overline{\psi}(\mu) \in \Hom^\mathcal{V}\left(\sigma^\vee \cap \mathcal{M}, \overline{\mathbb{Q}} \right)$ to be infinite on $(\sigma^\vee \setminus \sigma^\perp) \cap \mathcal{M}$ and to act on $f_1^{a_1}\cdots f_r^{a_r}g_1^{b_1}\cdots g_m^{b_m} \in \sigma^\perp \cap \mathcal{M}$ as follows:
\[
\overline{\psi}(\mu)\left( f_1^{a_1}\cdots f_r^{a_r}g_1^{b_1}\cdots g_m^{b_m} \right) = \min\set{\mu\left( f_{1\omega(1)}^{a_1}\cdots f_{r\omega(r)}^{a_r}g_1^{b_1}\cdots g_m^{b_m} \right) : 
\omega \in \Omega
}.
\]

\begin{remark}\label{allofOmega}
There is tacit assumption here that we only consider $\omega \in \Omega$ such that 
\[
f_{1\omega(1)}^{a_1}\cdots f_{r\omega(r)}^{a_r}g_1^{b_1}\cdots g_m^{b_m} \in \sigma_\mathfrak{a}^\vee \cap M
\] to ensure that $\mu$ is well-defined.
When $\mu = \nu_\gamma$ is induced by a $\mathbb{C}\{\!\{t\}\!\}$-point $\gamma$, then the minimum may be taken over the entirety of $\Omega$. An explanation of this appears in the proof of Theorem \ref{extendedmap}.
\end{remark}

\begin{example}
We will extend Example \ref{punctured affine in toric}. Recall that in this setting $r = 1$, $f_{11} = y$, and $f_{12} = x$. Consider the embedding $\Bl_0\left(\mathbb{C}^2 \right)$ of $G/H = \mathbb{C}^2 \setminus \set{0}$ given by a single non-colored ray in the direction of $v_1$, which we call $\sigma$.
The associated toric variety $Z$ is also $\Bl_0\left(\mathbb{C}^2 \right)$, given by the fan shown in Figure \ref{Blow up fans}.
\begin{figure}[H]
\begin{tikzpicture}
\draw[thick,->] (-5,1)--(-3,1);
\draw[fill] (-5,1) circle(.05);
\draw (-4,.75) node {$\sigma$};

\draw[gray!25,fill=gray!25] (0,0) rectangle (2,2);
\draw[thick,->] (0,0)--(2,0);
\draw[thick,->] (0,0)--(0,2);
\draw[thick,->] (0,0)--(2,2);
\draw[fill] (0,0) circle(.05); 

\draw (1,-.25) node {$v_{12}$};
\draw (-.4,1) node {$v_{11}$};
\draw (2.3,1.8) node {$\sigma$};
\end{tikzpicture}
\caption{Fans for $\Bl_0\left(\mathbb{C}^2 \right)$ as a spherical embedding (left) and a toric embedding (right)}
\label{Blow up fans}
\end{figure}
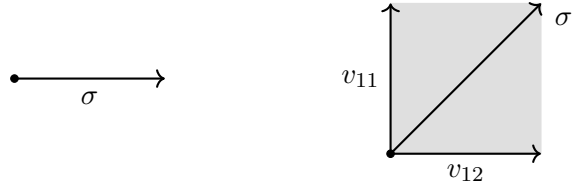
In the vector space $N_\mathbb{Q}$ corresponding to the toric variety, $\sigma$ is the ray in the fan of $\Bl_0\left(\mathbb{C}^2 \right)$ corresponding to the exceptional divisor.
There are three possibilities for $\mathfrak{a}$: $\mathfrak{a} = \emptyset$, $\set{v_{11}}$, or $\set{v_{12}}$.
These respectively correspond to three cones $\sigma_\mathfrak{a}$: the diagonal ray $\sigma$, the two-dimensional cone spanned by $\sigma$ and $v_{11}$, and the two-dimensional cone spanned by $\sigma$ and $v_{12}$.

The spherical tropicalization of $\Bl_0(\mathbb{C}^2)$ is isomorphic to $\overline{\mathbb{Q}}$, where the tropicalization of $\mathcal{E}$ is the point $\infty$.
In other words, $\sigma^\perp \cap \mathcal{M}$ is the origin in $\mathcal{N}$ and corresponds to the map $\mu \in \Hom\left( \sigma^\vee \cap \mathcal{M}, \overline{\mathbb{Q}} \right)$ sending every non-zero element of $\sigma^\vee \cap \mathcal{M}$ to $\infty$, i.e. the extended $G$-invariant map $\mathbb{C}[x,y] \rightarrow \overline{\mathbb{Q}}$ that is $\infty$ on all non-constant functions.

Therefore, if $\mu \in \sigma_\mathfrak{a}^\vee \cap M$ for some choice of $\mathfrak{a}$, then it must by definition of $\overline{\psi}$ be sent to $\infty$ in $\overline{\mathbb{Q}}$.
This can be viewed as collapsing the diagonally-oriented one-dimensional vector space and the two zero-dimensional vector spaces in $\trop_\mathbb{T}\left( \Bl_0\left(\mathbb{C}^2 \right) \right)$ to a point. See Figure \ref{Blow up trops} for reference.

\begin{figure}[H]
\begin{tikzpicture}
\draw[gray!25,fill=gray!25] (-10,0)--(-8,0)--(-8,1)--(-9,2)--(-10,2);
\draw[thick] (-8,1)--(-8,0);
\draw[thick] (-9,2)--(-10,2);
\draw[thick] (-9,2)--(-8,1);
\draw[fill] (-9,2) circle(.05); 
\draw[fill] (-8,1) circle(.05);

\draw[gray!25,fill=gray!25] (-5,0)--(-3,0)--(-3,1)--(-4,2)--(-5,2);
\draw[thick] (-3,1)--(-3,0);
\draw[thick] (-4,2)--(-5,2);
\draw[thick] (-4,2)--(-3,1);
\draw[fill] (-4,2) circle(.05); 
\draw[fill] (-3,1) circle(.05);

\draw[thick,dotted] (-5,0)--(-3.5,1.5);
\draw[fill] (-3.5,1.5) circle(.05);

\draw[thick,->, dotted] (-3,.25)--(-4.6,.25);
\draw[thick,->, dotted] (-3,.75)--(-4.1,.75);
\draw[thick,->, dotted] (-4.75,2)--(-4.75,.4);
\draw[thick,->, dotted] (-4.25,2)--(-4.25,.9);

\draw[thick,->, dotted] (-2.875,1.125)--(-3.25,1.5);
\draw[thick,->, dotted] (-3.875,2.125)--(-3.5,1.75);

\draw[thick] (0,0)--(1.5,1.5);
\draw[fill] (1.5,1.5) circle(.05);

\draw (-9,-.5) node {$\trop_\mathbb{T}\left( \Bl_0\left( \mathbb{C}^2 \right) \right)$};
\draw (-4,-.5) node {Action of $\overline{\psi}$};
\draw (1,-.5) node {$\trop_G\left( \Bl_0\left( \mathbb{C}^2 \right) \right)$};

\end{tikzpicture}
\caption{$\overline{\psi}\left( \trop_\mathbb{T}\left( \Bl_0\left( \mathbb{C}^2 \right) \right) \right) = \trop_G\left( \Bl_0\left( \mathbb{C}^2 \right) \right)$}
\label{Blow up trops}
\end{figure}
\end{example}

In the previous example we saw that $\overline{\psi}$ takes the extended toric tropicalization to the extended spherical tropicalization.
Our goal now is to prove that this holds in general. 
To accomplish this, we need an additional theorem:

\begin{theorem}[\cite{MS}, Theorem 6.2.18]\label{closurecommutestoric}
Let $Y \subseteq \mathbb{T}$, and let $\overline{Y}$ be the closure of $Y$ in a toric variety $Z$. Then
\[
\trop_\mathbb{T}\left( \overline{Y} \right) = \overline{\trop_\mathbb{T}(Y)}.
\]
\end{theorem}

\begin{theorem}\label{extendedmap}
If $G/H$ has trivial divisor class group, $Y \subseteq G/H$ is a subvariety, and $X$ is a $G/H$-embedding with the $A_2$-property, then 
\[
\overline{\psi}\left( \trop_{\mathbb{T}}\left(\overline{Y}\right) \right) = \trop_G\left( \overline{Y} \right),
\]
where the closure on the left is taken in $Z$ and the closure on the right is taken in $X$.
\end{theorem}
\begin{proof}
Let $\gamma: \spec{\mathbb{C}\{\!\{t\}\!\}} \rightarrow \overline{Y}$ be a $\mathbb{C}\{\!\{t\}\!\}$-point of $\overline{Y}$ and suppose $\tilde{\nu}_\gamma$ and $\nu_\gamma$ are respectively the $\mathbb{T}$-invariant and $G$-invariant extended valuations induced by $\gamma$.
Then $\gamma$ lies in a $G$-orbit corresponding to a colored cone $(\sigma,\mathcal{F})$ and a $\mathbb{T}$-orbit corresponding to a cone $\sigma_\mathfrak{a}$.
We will show that $\overline{\psi}\left(\tilde{\nu}_\gamma\right) = \nu_\gamma$.

Let $f_1^{a_1}\cdots f_r^{a_r}g_1^{b_1} \cdots g_m^{b_m} \in \sigma^\perp \cap \mathcal{M}$.
Then a sufficiently general element of $G$ takes this $B$ semi-invariant rational function to a function of the form
\[
(c_{11}f_{11} + \cdots + c_{1s_1}f_{1s_1})^{a_1} \cdots (c_{r1}f_{r1} + \cdots + c_{rs_r}f_{rs_r})^{a_r}\cdot cg_1^{b_1} \cdots g_m^{b_m}
\]
for $c,c_{ij} \in \mathbb{C}^*$.
Generically, this is a non-zero rational function on the $G$-orbit corresponding to $(\sigma, \mathcal{F})$ that is defined at $\gamma$, so the valuation $\tilde{\nu}_\gamma$ takes it to a finite rational number.
By Theorem \ref{closurecommutestoric}, $\tilde{\nu}_\gamma$ lies in the closure of $\trop_\mathbb{T}(Y)$, so there exists a sequence $\set{\nu_\ell}_{\ell=1}^\infty$ of $\mathbb{T}$-invariant valuations associated to $\mathbb{C}\{\!\{t\}\!\}$-points $\gamma_\ell$ of $Y$ such that $\lim_{\ell \rightarrow \infty} \nu_\ell = \tilde{\nu}_\gamma$ in the topology on $\trop_\mathbb{T}(Z)$.
Then we have the following, which verifies the claim:
\begin{align*}
& \nu_\gamma\left( f_1^{a_1}\cdots f_r^{a_r}g_1^{b_1} \cdots g_m^{b_m} \right) \\
= & \; \tilde{\nu}_\gamma\left((c_{11}f_{11} + \cdots + c_{1s_1}f_{1s_1})^{a_1} \cdots (c_{r1}f_{r1} + \cdots + c_{rs_r}f_{rs_r})^{a_r}\cdot cg_1^{b_1} \cdots g_m^{b_m} \right) \\
= & \lim_{\ell \rightarrow \infty} \nu_\ell\left((c_{11}f_{11} + \cdots + c_{1s_1}f_{1s_1})^{a_1} \cdots (c_{r1}f_{r1} + \cdots + c_{rs_r}f_{rs_r})^{a_r}\cdot cg_1^{b_1} \cdots g_m^{b_m} \right) \\
= & \lim_{\ell \rightarrow \infty} \sum_{i=1}^r a_i\min_j\set{\nu_\ell(f_{ij})} + \sum_{k=1}^m b_k\nu_\ell(g_k) \\
= & \lim_{\ell \rightarrow \infty} \min\set{\nu_\ell\left( f_{1\omega(1)}^{a_1}\cdots f_{r\omega(r)}^{a_r}g_1^{b_1}\cdots g_m^{b_m} \right) : 
\omega \in \Omega} \\
= & \min\set{\tilde{\nu}_\gamma\left( f_{1\omega(1)}^{a_1}\cdots f_{r\omega(r)}^{a_r}g_1^{b_1}\cdots g_m^{b_m} \right) : 
\omega \in \Omega} = \overline{\psi}\left(\tilde{\nu}_\gamma \right)\left( f_1^{a_1}\cdots f_r^{a_r}g_1^{b_1} \cdots g_m^{b_m} \right)
\end{align*}
Pursuant to Remark \ref{allofOmega}, note that the minimum here is indexed over the entirety of the set $\Omega$ since every $c_{ij}$ is nonzero.
\end{proof}

Having established the result when the homogeneous space has trivial divisor class group, we turn to the case of non-trivial divisor class group.
As before, we write $\bm{G}/\bm{H}$ to denote a homogeneous space with non-trivial divisor class group and write $G/H$ for the homogeneous space with trivial divisor class group and dominant map $\bm{\pi}: G/H \rightarrow \bm{G}/\bm{H}$.
If $X$ is the $G/H$-embedding associated to a $\bm{G}/\bm{H}$-embedding $\bm{X}$, then the extended map $\bm{\pi}: X \rightarrow \bm{X}$ is equivariant with respect to the natural surjection $G \rightarrow \bm{G}$.
We can therefore extend Theorem \ref{V=G2} to spherical embeddings of homogeneous spaces with non-trivial divisor class group; the proof proceeds almost identically.
\begin{theorem}\label{nontrivial extended map}
Let $\bm{G}/\bm{H} \hookrightarrow \bm{X}$ be a spherical embedding with associated spherical embedding $G/H \hookrightarrow X$ where $G/H$ has trivial divisor class group.
If $\bm{Y} \subseteq \bm{G}/\bm{H}$ is a closed subvariety and $\bm{X}$ is a $\bm{G}/\bm{H}$-embedding, then
\[
\trop_{\bm{G}}\left( \overline{\bm{Y}} \right) = \left(\trop(\bm{\pi}) \circ \overline{\psi} \right)\left( \trop_{\bm{\mathbb{T}}}\left(\bm{\pi}^{-1}\left( \overline{\bm{Y}} \right) \right) \right),
\]
where the closures are taken in $\bm{X}$.
\end{theorem}
\begin{proof}
By Theorem \ref{extendedmap}, this is equivalent to showing 
\[
\trop_{\bm{G}}\left( \overline{\bm{Y}} \right) = \trop(\bm{\pi})\left( \trop_G\left( \bm{\pi}^{-1} \left( \overline{\bm{Y}} \right) \right) \right).
\]
Because $\bm{\pi}$ is equivariant with respect to the surjective group homomorphism $G \rightarrow \bm{G}$, the statement follows from Proposition \ref{commute}.
\end{proof}

Developing the interplay between spherical tropicalization and toric tropicalization has the potential to give insight into the structure of spherical tropicalizations by translating known results from the toric world.
As an example, we can generalize Theorem \ref{closurecommutestoric} to spherical varieties:

\begin{theorem}\label{closurecommutesspherical}
If $Y \subseteq G/H$ is a closed subvariety and $X$ is a $G/H$-embedding, then 
\[
\trop_G\left(\overline{Y} \right) = \overline{\trop_G(Y)},
\]
where the closure on the left is taken in $X$ and the closure on the right is taken in $\trop_G(X)$.
\end{theorem}
\begin{proof}
We prove the theorem when $X$ is a simple $G/H$-embedding because $X$ always has the $A_2$-property in this setting. If $X$ is not simple, we may break it up into simple $G/H$-embeddings, where the result holds, and then glue together along shared orbits.
Note that this will work even if $X$ does not have the $A_2$-property.

We suppose first that $G/H$ has trivial divisor class group.
Let $\nu_\gamma \in \trop_G\left(\overline{Y} \right) \setminus \trop_G(Y)$ be an extended $G$-invariant valuation corresponding to a $\mathbb{C}\{\!\{t\}\!\}$-point $\gamma$.
Then $\gamma$ induces a $\mathbb{T}$-invariant valuation $\tilde{\nu}_\gamma$ in $\trop_\mathbb{T}\left(\overline{Y} \right)$.
By Theorem \ref{closurecommutestoric}, there exists a sequence of $\mathbb{C}\{\!\{t\}\!\}$-points $\gamma_\ell \in Y\left(\mathbb{C}\{\!\{t\}\!\} \right)$ with associated $\mathbb{T}$-invariant valuations $\tilde{\nu}_\ell$ such that $\lim_{\ell \rightarrow \infty} \tilde{\nu}_\ell = \tilde{\nu}_\gamma$ in the topology on $\trop_\mathbb{T}(Z)$.
We claim that $\lim_{\ell \rightarrow \infty} \overline{\psi}\left(\tilde{\nu}_\ell \right) = \nu_\gamma$.
Suppose $\gamma$ lies in a $G$-orbit corresponding to a colored cone $(\sigma, \mathcal{F})$ and a $\mathbb{T}$-orbit corresponding to $\sigma_\mathfrak{a}$ and let $f_1^{a_1}\cdots f_r^{a_r}g_1^{b_1} \cdots g_m^{b_m} \in \sigma^\perp \cap \mathcal{M}$ be arbitrary.
Then we have the following, recalling from Remark \ref{allofOmega} that the minimums are indexed over all of $\Omega$:
\begin{align*}
\lim_{\ell \rightarrow \infty} \overline{\psi}(\nu_\ell)\left(f_1^{a_1}\cdots f_r^{a_r}g_1^{b_1} \cdots g_m^{b_m} \right) & = \lim_{\ell \rightarrow \infty} \min\set{\tilde{\nu}_\ell\left( f_{1\omega(1)}^{a_1}\cdots f_{r\omega(r)}^{a_r}g_1^{b_1}\cdots g_m^{b_m} \right) : 
\omega \in \Omega} \\
& = \min\set{\tilde{\nu}_\gamma\left( f_{1\omega(1)}^{a_1}\cdots f_{r\omega(r)}^{a_r}g_1^{b_1}\cdots g_m^{b_m} \right) : 
\omega \in \Omega} \\
& = \overline{\psi}\left( \tilde{\nu}_\gamma \right)\left( f_1^{a_1}\cdots f_r^{a_r}g_1^{b_1} \cdots g_m^{b_m} \right) \\
& = \nu_\gamma \left( f_1^{a_1}\cdots f_r^{a_r}g_1^{b_1} \cdots g_m^{b_m} \right).
\end{align*}
The last equality here follows from the proof of Theorem \ref{extendedmap}.
It follows that $\nu_\gamma$ is the limit of the sequence $\set{\overline{\psi}\left(\tilde{\nu}_\ell \right)}_{\ell = 1}^\infty$.
Hence, $\nu_\gamma \in \overline{\trop_G(Y)}$ and the first inclusion follows.

Conversely, suppose that $\mu \in \overline{\trop_G(Y)}$ and let $\set{\gamma_\ell}_{\ell = 1}^\infty$ be a sequence of $\mathbb{C}\{\!\{t\}\!\}$-points of $Y$ such that the associated $G$-invariant valuations $\mu_\ell := \nu_{\gamma_\ell}$ satisfy $\lim_{\ell \rightarrow \infty} \mu_\ell = \mu$ in the topology on $\trop_G(X)$.
Each $\gamma_\ell$ induces a $\mathbb{T}$-invariant valuation $\tilde{\mu}_\ell$.
By possibly replacing $\set{\mu_\ell}_{\ell = 1}^\infty$ with a subsequence, there is $\tilde{\mu} \in \overline{\trop_\mathbb{T}\left( Y \right)}$ with $\tilde{\mu} = \lim_{\ell \rightarrow \infty} \tilde{\mu}_\ell$.
By Theorem \ref{closurecommutestoric}, $\tilde{\mu}$ is induced by a $\mathbb{C}\{\!\{t\}\!\}$-point $\gamma$ of $\overline{Y}$.
We claim that $\mu = \overline{\psi}\left( \tilde{\mu} \right) \in \trop_G\left( \overline{Y} \right)$.
Let $f_1^{a_1} \cdots f_r^{a_r}g_1^{b_1} \cdots g_m^{b_m} \in \sigma^\perp \cap \mathcal{M}$ be arbitrary, where $\sigma$ corresponds to the $G$-orbit in whose tropicalization $\mu$ lies.
Then the equality
\[
\mu\left( f_1^{a_1} \cdots f_r^{a_r}g_1^{b_1} \cdots g_m^{b_m} \right) = \overline{\psi}\left( \tilde{\mu} \right)\left( f_1^{a_1} \cdots f_r^{a_r}g_1^{b_1} \cdots g_m^{b_m} \right)
\]
follows from an argument similar to the previous inclusion.
Hence, $\mu$ is induced by $\gamma$ and so $\mu \in \trop_G\left( \overline{Y} \right)$.

Now suppose $\bm{G}/\bm{H}$ has nontrivial divisor class group and let $G/H$ be the associated spherical homogeneous space with trivial divisor class group.
Then we have the following string of inclusions and equalities:
\begin{align*}
\overline{\trop_{\bm{G}}\left( \bm{Y} \right)} & = \overline{\trop(\bm{\pi})\left( \trop_G\left( \bm{\pi}^{-1} \left(\bm{Y} \right) \right) \right)} \\
& \supseteq 
\trop(\bm{\pi})\left( \overline{\trop_G\left( \bm{\pi}^{-1} \left(\bm{Y} \right) \right)} \right) \\
& = \trop(\bm{\pi})\left( \trop_G\left( \overline{\bm{\pi}^{-1} \left(\bm{Y} \right)} \right) \right) \\
& \subseteq \trop(\bm{\pi})\left( \trop_G\left( \bm{\pi}^{-1} \left(\overline{\bm{Y}} \right) \right) \right) \\
& = \trop_{\bm{G}}\left( \overline{\bm{Y}} \right)
\end{align*}
From top to bottom, these containments and equalities respectively follow from Theorem \ref{V=G2}, Proposition \ref{trop map continuous}, the first half of this proof, the continuity of $\bm{\pi}$, and Theorem \ref{nontrivial extended map}.
Completing the proof now comes down to showing that the two containments are equalities.

For the first containment, suppose that $\mu \in \overline{\trop(\bm{\pi})(W)}$, where $W := \trop_G\left( \bm{\pi}^{-1} \left(\bm{Y} \right) \right)$.
Then there exists a sequence $\set{\mu_\ell}_{\ell = 1}^\infty$ of valuations in $\trop(\bm{\pi})\left(W \right)$ such that $\lim_{\ell \rightarrow \infty} \mu_\ell = \mu$.
For each $\ell$, choose an element $\bm{\mu}_\ell \in\trop(\bm{\pi})^{-1}(\mu_\ell)$.
The set $\set{\bm{\mu}_\ell}_{\ell = 1}^\infty$ contains a convergent subsequence in $W$.
This subsequence converges to an extended valuation $\tilde{\mu} \in \overline{W}$, which necessarily maps to $\mu$ under $\trop(\bm{\pi})$.
Thus, the first inclusion is equality.  

Now we consider the second inclusion.
We will show that $\trop_\mathbb{T}\left( \overline{\bm{\pi}^{-1} \left(\bm{Y} \right)} \right) \supseteq \trop_\mathbb{T}\left( \bm{\pi}^{-1} \left(\overline{\bm{Y}} \right) \right)$ since applying the map $\trop(\bm{\pi}) \circ \overline{\psi}$ will deliver the needed inclusion by Theorem \ref{nontrivial extended map}.
Theorem \ref{closurecommutestoric} says that $\overline{\trop_\mathbb{T}\left( \bm{\pi}^{-1} \left(\bm{Y} \right) \right)} = \trop_\mathbb{T}\left( \overline{\bm{\pi}^{-1} \left(\bm{Y} \right)} \right)$, and the proof in \cite{MS} only uses the fact that $\overline{\bm{\pi}^{-1} \left(\bm{Y} \right)}$ is a closed set in $Z$ whose intersection with $\mathbb{T}$ is the variety $\bm{\pi}^{-1} \left(\bm{Y} \right) \cap \mathbb{T}$.
Because $\bm{\pi}^{-1}\left( \overline{\bm{Y}} \right)$ satisfies these properties, we can conclude that $\overline{\trop_\mathbb{T}\left( \bm{\pi}^{-1} \left(\bm{Y} \right) \right)} = \trop_\mathbb{T}\left(\bm{\pi}^{-1} \left(\overline{\bm{Y}} \right) \right)$ and hence the equality $\trop_\mathbb{T}\left( \overline{\bm{\pi}^{-1} \left(\bm{Y} \right) }\right) = \trop_\mathbb{T}\left(\bm{\pi}^{-1} \left(\overline{\bm{Y}} \right) \right)$.
This completes the second inclusion and the proof.
\end{proof}


\end{document}